\documentclass[sn-mathphys,Numbered]{sn-jnl}

\usepackage{graphicx}%
\usepackage{multirow}%
\usepackage{amsmath,amssymb,amsfonts}%
\usepackage{amsthm}%
\usepackage{mathtools}%
\usepackage{mathrsfs}%
\usepackage[title]{appendix}%
\usepackage{xcolor}%
\usepackage{textcomp}%
\usepackage{manyfoot}%
\usepackage{booktabs}%
\usepackage{algorithm}%
\usepackage{algorithmic}%
\usepackage{listings}%

\usepackage{enumitem}
\usepackage{parskip}

\setlist[enumerate]{parsep=\smallskipamount}

\newtheorem{theorem}{Theorem}
\newtheorem{lemma}{Lemma} 
\newtheorem{proposition}{Proposition}

\newtheorem{example}{Example}%
\newtheorem{remark}{Remark}%
\newtheorem*{remarknot}{Remark}%

\newtheorem{definition}{Definition}%

\renewcommand{\epsilon}{\varepsilon}
\newcommand{\R}{\mathcal{R}}
\newcommand{\Rb}{\mathbb{R}}
\newcommand{\nR}{\nabla \mathcal{R}}
\newcommand{\critic}{\mathcal{C}}
\newcommand{\statio}{\mathcal{E}}
\newcommand{\zerodot}{\mathcal{Z}}
\newcommand{\accum}{\mathcal{A}}
\newcommand{\voisi}{\mathcal{V}}
\newcommand{\globalset}{\mathcal{G}}

\newcommand{\Frac}[2]{\displaystyle \frac{#1}{#2}\otimes }

\begin{document}

\title[An Abstract Lyapunov Control Optimizer: Local Stabilization and Global Convergence]{An Abstract Lyapunov Control Optimizer: Local Stabilization and Global Convergence}


\author*[1,2]{\fnm{Bilel} \sur{Bensaid}}\email{bilel.bensaid@u-bordeaux.fr}

\author[1]{\fnm{Ga\"el} \sur{Po\"ette}}\email{gael.poette@cea.fr}
\equalcont{These authors contributed equally to this work.}

\author[2]{\fnm{Rodolphe} \sur{Turpault}}\email{rodolphe.turpault@u-bordeaux.fr}
\equalcont{These authors contributed equally to this work.}

\affil*[1]{\orgdiv{CESTA-DAM}\orgname{/CEA}, \orgaddress{\city{Le Barp}, \postcode{33114},\country{France}}}

\affil[2]{\orgdiv{Institut de Mathématiques de Bordeaux, Université de Bordeaux}, \orgname{CNRS,Bordeaux INP}, \orgaddress{\city{Talence}, \postcode{33400}, \country{France}}}


\abstract{Recently, many machine learning optimizers have been analysed considering them as the asymptotic limit of some differential equations\cite{variational_perspective} when the step size goes to zero. In other words, the optimizers can be seen as a finite difference scheme applied to a continuous dynamical system. But the major part of the results in the literature concerns constant step size algorithms. The main aim of this paper is to investigate the guarantees of the adaptive step size counterpart. In fact, this dynamical point of view can be used to design step size update rules, by choosing a discretization of the continuous equation that preserves its most relevant features \cite{Bilel, control1}. In this work, we analyse this kind of adaptive optimizers and prove their Lyapunov stability and convergence properties for any choice of hyperparameters. At the best of our knowledge, this paper introduces for the first time the use of continuous selection theory from general topology to overcome some of the intrinsic difficulties due to the non constant and non regular step size policies. The general framework developed gives many new results on adaptive and constant step size Momentum/Heavy-Ball \cite{Polyak} and p-GD\cite{cv_qflow} algorithms.}




\maketitle

\section{Introduction}\label{intro}

Many Machine Learning tasks require to minimize a differentiable non-convex objective function $\R(\theta)$ defined on $\Rb^N$ where $N$ is usually the number of parameters of the model. The literature is strewn with algorithms aiming at reaching the previous goal with a range of
complexity strongly depending on the operational constraints we face (is the gradient easily available? The Hessian? The quasi-Hessian? etc.). In the following lines, we assume that $\R$ and $\nabla \mathcal{R}$ are available. This is common for many applications, notably in Neural Networks learning \cite{backpropagation}. 
Certainly, the most simple algorithm to minimize $\mathcal{R}$ is the \textbf{gradient descent (GD)} with a constant step size $\eta$ called also the learning rate (which is an hyperparameter to tune):
\begin{equation*}
\theta_{n+1} = \theta_n -\eta \nR(\theta_n).
\end{equation*}
To enhance its speed, more complex algorithms have been designed: for example, a first family of algorithms is based on adding a memory on the gradient. This gives the \textbf{Momentum algorithm also known as the Heavy-Ball optimizer} \cite{Polyak}:
\begin{equation*}
    \left\{
    \begin{array}{l}
        v_{n+1} = (1-\beta_1)v_{n} + \beta_1 \nabla \mathcal{R}(\theta_n), \\
        \theta_{n+1} = \theta_n - \eta v_{n+1},
    \end{array}
    \right.
    \label{defMomentum}
\end{equation*}
where $v_0=0$ and $\beta_1$ is a hyperparameter lying in $]0,1[$ (which may need to be tuned).
Other algorithms keep in memory the square of the gradient, like \textbf{RMSProp} \cite{RMSProp} which is widely used in Deep Learning, and whose implementation is given by:
\begin{equation*}
    \left\{
    \begin{array}{l}
        s_{n+1} = (1-\beta_2)s_n + \beta_2 \left(\nabla \mathcal{R}(\theta_n)\right)^{{\otimes} 2}, \\
        \theta_{n+1} = \theta_n -\eta \Frac{\nR(\theta_n)}{\sqrt{{s}_{n+1}+\varepsilon_a}},
    \end{array}
    \right.
    \label{defRMSProp}
\end{equation*}
where $\beta_2$ is a hyperparameter, like $\beta_1$ for Momentum. In \eqref{defRMSProp}, the operations must be understood componentwise (for the square and the division). The small parameter $\varepsilon_a$, generally set between $10^{-7}$ and $10^{-15}$, prevents division by zero.\\
Momentum and RMSProp are part of the general family of inertial algorithms like AdaGrad \cite{Adagrad} and Adam \cite{Adam}. This family is also called accelerated gradients because they are faster for convex functions and achieve the optimal convergence rate in this class of functions \cite{Nesterov,Nesterov_acceleration_ODE}. In addition to inertial algorithms, a second family based on re-scaling strategies, aims at tackling the problem of exploding gradients. Exploding gradients refer to situations in which the gradients stiffen along the training, making the iterative process unstable \cite{RNN_difficult}. For such type of strategy, the gradient is replaced by its normalization $\dfrac{\nR(\theta)}{\|\nR(\theta)\|}$ when it exceeds a certain threshold. This simple technique is called {\em gradient clipping or normalized gradient} and is widely used for training Recurrent Neural Networks \cite{RNN_difficult}. Recently a more general class of rescaled gradient descent have been suggested in \cite{variational_perspective} (theorem 3.4). The iterative process is given by, for $p>1$:
\begin{equation}
    \theta_{n+1} = \theta_n - \eta \dfrac{\nR(\theta_n)}{\|\nR(\theta_n)\|^{\frac{p-1}{p-2}}} \text{ if } \nR(\theta_n) \neq 0.
    \label{p_GD}
\end{equation}
It is called \textbf{p-gradient flow or p-gradient descent (pGD)}. The normalized gradient is obtained when $p=+\infty$. \\

The previous algorithms are often called Machine Learning(ML) {\em optimizers} in the literature. The short list above is not exhaustive but is sufficient to illustrate the results of this paper. For these optimizers, the convergence results, for the iterative process to end in a vicinity of a minimum of $\R$, are mainly limited to the convex case or to Lipshitz gradient function \cite{Nesterov, Lyapunov_Nesterov,Nesterov_acceleration_ODE}. Even in this last configuration, the results for GD hold for a time step satisfying $\eta \leq \frac{1}{L}$ where $L$ is the Lipshitz constant of $\nR$. In many optimization problems particularly when Neural Networks are involved, even if $\nR$ is Lipshitz continuous, the constant $L$ is not available \cite{Lipshitz_constant} and the limitation on the time step is of little practical use. On another hand, adaptive time step strategies are more practical and represent a third family of strategies. These consist in looking for time step $\eta_n$ satisfying the so-called descent condition or Armijo condition \cite{armijo}:
\begin{equation*}
    \R(\theta_n-\eta_n\nR(\theta_n))-\R(\theta_n) \leq -\lambda \eta_n \|\nR(\theta_n)\|^2.
\end{equation*}
There are some recent works about the convergence of the gradient algorithm (or more generally descent algorithms) under the Armijo condition for analytical functions \cite{Absil,Rondepierre}.

It is not easy to identify which of the previous algorithms (GD, Momentum, RMSProp, GD under Armijo conditions...)  is superior to the others: some are more complex, some require more computations, some need to store more quantities into memory, etc. Their capabilities can be, in a way, ranked from the properties one can expect from them. For this, we need to thoroughly analyse them. Interestingly, many optimizers can be analysed as the
discretization of an ordinary differential equation (ODE):
\begin{equation*}
y'(t) = F(y(t)),
\label{ODE}
\end{equation*}
where $F:\Rb^m \mapsto \Rb^m$ is generally considered continuous. For instance, GD can be interpreted as the explicit Euler discretization of the following flow/ODE equation:
\begin{equation*}
        \left\{
        \begin{array}{l}
        \theta(0)=\theta_0, \\
        \theta'(t) = -\nabla \mathcal{R}(\theta(t)).
        \end{array}
        \right.
        \label{ODE_GD}
\end{equation*}
Here, $m=N$, $y=\theta$ and $F(\theta)=-\nR(\theta)$.
In a similar manner, Momentum asymptotically solves the following damping system of ODEs \cite{Polyak, momentum_lyapunov, momentum_odes}:
\begin{equation*}
\left\{
  \begin{array}{ll}
       v'(t)=-\bar{\beta_1}v(t)-\bar{\beta_1}\nR(\theta(t)),  \\
       \theta'(t) = v(t),
  \end{array}
\right.
  \label{ODE_Momentum_order2}
\end{equation*}
with initial conditions $\theta(0)=\theta_0$ and $\theta'(0)=0$ where $\bar{\beta_{1}} = \frac{\beta_1}{\eta}$. For Momentum, $m=2N$, $y=(v,\theta)^T$ and:
\begin{equation*}
    F(v,\theta) = 
    \begin{pmatrix}
          -\bar{\beta_1}v-\bar{\beta_1}\nR(\theta) \\
          v
    \end{pmatrix}
    .
\end{equation*}

A powerful tool for the qualitative analysis of the ODE is Lyapunov theory \cite{Lyapunov}. Let us recall that a Lyapunov function is a functional $V:\Rb^m \mapsto \Rb^+$ that decreases along the trajectories of the ODE. More formally, its time derivative is negative:
\begin{equation*}
        \dot{V}(y) \coloneqq \nabla V (y)^T F (y)\leq 0.
\end{equation*}
This time derivative is often considered as a dissipation rate since $V$ can be seen as the total energy of the system when the ODE has a physical interpretation. The existence of such a Lyapunov function $V$ for a flow of interest gives many insights on what can be \textit{asymptotically} expected from the optimizers (i.e. as $\eta \rightarrow 0$): local stability and convergence of the trajectories (see \cite{Nesterov, variational_perspective, CV_Adam,cv_qflow} for a non-exhaustive list of examples). Indeed, one does not expect more properties for the discrete flow than for the continuous one.\\
In a recent work \cite{Bilel}, it has been noted that these \textit{asymptotical} properties are not enough in practice: these algorithms, with some common choices of hyperparameters, can generate surprising instabilities (which are not supposed to occur asymptotically). 
The risks, when coping with these instabilities, can be summed up as divergences of the optimization process, jumps in the vicinities of global/local minima and so on, for an overall loss of computational time or of performance/accuracy.
In \cite{Bilel}, Lyapunov functions $V$ of several continuous flows of some optimizers have been identified and used during the discrete process: by selecting the time step {$\eta$} in order to decrease $V$ in the discrete field, i.e. such that $V(y_{n+1})-V(y_n) \leq 0$, the authors have empirically observed a stabilization of the trajectories. One of the aims of this paper is provide some theoretical justification to go beyond these observations and the use of the Lyapunov function $V$ during the (discrete) optimization. 

In \cite{Bilel}, the cornerstone of the analysis is the decreasing of $V$ but it seems natural to also preserve the dissipation rate/speed of the continuous equation in order to deduce more qualitative properties. So we may want to enforce the equality $V(y_{n+1})-V(y_n) = \eta_n \dot{V}(y_n)$ which can be viewed as a first order approximation of $\dot{V}$:
\begin{equation*}
    \displaystyle{\lim_{\eta \to 0}} \dfrac{V\bigl(y+\eta F(y)\bigr)-V(y)}{\eta} = \dot{V}(y).
\end{equation*}
Enforcing the equality is however difficult in practice and leads to the resolution of many non-linear systems: see for example the projection methods in \cite{projective_lyapunov}. 
Therefore we here choose to preserve a weak form of the dissipation rate up to a constant: 
\begin{equation}
V(y_{n+1})-V(y_n) \leq \lambda \eta_n \dot{V}(y_n),
\label{dissipation_inequality}
\end{equation}
where $\lambda$ is a real hyperparameter in $]0,1[$.
This can be seen as a \textbf{generalization to arbitrary optimizers (with an identified Lyapunov function) of the Armijo condition} defined for descent methods like GD. \\
From now on, it remains to discretize the ODE \eqref{ODE} while respecting the qualitative (or physical) constraint \eqref{dissipation_inequality}. Let us present two practical implementations aiming at enforcing this inequality with algorithms \ref{algo_LCR} (called LCR as Lyapunov Control Restart) and \ref{algo_LCM} (called LCM as Lyapunov Control Memory).

\begin{algorithm}[ht]
   \caption{Optimization by Lyapunov Control with time step restart (LCR)}
   \begin{algorithmic}
       \REQUIRE initial values $y_0$, $\eta_{init}$, $f_1>1$ and $\epsilon>0$.
       
 \STATE $\dot{V} \leftarrow \dot{V}(y_0)$
       
       \WHILE{$|\dot{V}|>\epsilon$}
       \STATE $V_0 \leftarrow V(y)$
        \STATE $\dot{V} \leftarrow \dot{V}(y)$
        \STATE $y_0 \leftarrow y$
        \REPEAT
            \STATE $y \leftarrow y + \eta F(y)$
            \STATE $V \leftarrow V(y)$
            \IF{$V-V_0>\lambda \eta \dot{V}$}
                \STATE $\eta \leftarrow \frac{\eta}{f_1}$
                \STATE $y \leftarrow y_0$
            \ENDIF
        \UNTIL{$V-V_0\leq \lambda \eta \dot{V}$}
        \STATE $\eta \leftarrow \eta_{init}$
        \STATE $n \leftarrow n+1$
       \ENDWHILE
   \end{algorithmic}
   \label{algo_LCR}
\end{algorithm}

\begin{algorithm}[ht]
   \caption{Optimization by Lyapunov Control with time step memory (LCM)}
   \begin{algorithmic}
       \REQUIRE initial values $y_0$, $\eta_{init}$, $f_1>1$, $f_2>1$ and 
       $\epsilon>0$.

 \STATE $\dot{V} \leftarrow \dot{V}(y_0)$
       
       \WHILE{$|\dot{V}|>\epsilon$}
       \STATE $V_0 \leftarrow V(y)$
        \STATE $\dot{V} \leftarrow \dot{V}(y)$
        \STATE $y_0 \leftarrow y$
        \REPEAT
            \STATE $y \leftarrow y + \eta F(y)$
            \STATE $V \leftarrow V(y)$
            \IF{$V-V_0>\lambda \eta \dot{V}$}
                \STATE $\eta \leftarrow \frac{\eta}{f_1}$
                \STATE $y \leftarrow y_0$
            \ENDIF
        \UNTIL{$V-V_0 \leq \lambda \eta \dot{V}$}
        \STATE $\eta \leftarrow f_2 \eta$
        \STATE $n \leftarrow n+1$
       \ENDWHILE
   \end{algorithmic}
   \label{algo_LCM}
\end{algorithm}

In both algorithms, we use an explicit Euler scheme to discretize the equation (for simplification) and we are looking for a time step that decreases the Lyapunov function at an appropriate rate. The reduction is done in practice thanks to the constant factor $f_1$. The algorithms can be summed up by the following constraint equation:
\begin{equation*}
        y_{n+1} = y_n + \eta_n F(y_n),
        \label{Euler}
\end{equation*}
where $\eta_n$ is chosen to verify \eqref{dissipation_inequality}.
The differences between the two algorithms may seem negligible: in algorithm LCR (algorithm\#\ref{algo_LCR}) the time step takes a fixed value $\eta_{init}$ before proceeding to a linesearch whereas in algorithm LCM (algorithm\#\ref{algo_LCM}), the previous time step is used multiplied by a factor $f_2$. Although this change seems insignificant, the LCM version seems more efficient in practice \cite{Rondepierre} but its analysis is much more challenging, this will be developped in the following sections.  
Similar backtracking algorithms were suggested for GD in \cite{Rondepierre}. More recently, the LCM version of the algorithm in its general form was proposed in the control field in \cite{control2} to discretize an ODE with one global equilibrium in order to preserve its asymptotic behavior. Contrary to constant time step algorithms, these optimizers with updating strategies (LCR and LCM) look for a time step lying in the set (resolution of an inequality at each iteration):
\begin{equation}
    I(y) = \{\eta>0, f(y,\eta)\leq 0\}.
    \label{I}
\end{equation}
where the function $f$ is defined on $\Rb^m \times \Rb_+$ as follows:
\begin{equation}
    f(y,\eta) = V(y+\eta F(y))-V(y)-\lambda \eta \dot{V}(y).
    \label{f}
\end{equation}

The {\bf aim of this paper is to provide guarantees for such updating strategies in the non-convex setting (multiple equilibriums)} and to show that the fundamental condition \eqref{dissipation_inequality} makes it possible to preserve several good features of the continuous time equation (ODE). \\
The paper is organized as follows. First, section \ref{limit_points} deals with the localization of the accumulation points of the sequence $(y_n)_{n \in \mathbb{N}}$ generated by algorithms LCR and LCM. In particular, we prove that they satisfy a weak version of the LaSalle's ODE principle \cite{LaSalle}. In this section, a fundamental and new tool for analysing adaptive optimizers is presented by applying \textbf{selection theory} for multi-applications. Then, in section \ref{section_stability}, we prove a discrete stability theorem with the same hypothesis as the classical Lyapunov theorem \cite{Lyapunov} for ODEs. These hypotheses are weakened compared to stability results for this kind of algorithm, see \cite{Absil}. Finally, section \ref{section_cv} presents a general convergence framework for these updating strategies with an interesting application to rescaled gradients. Through the different sections, the theoretical results are illustrated on some classical machine learning optimizers.  

\section{The difficulty of the limit points}
\label{limit_points}

In this section, we investigate on the set in which the limit of the sequence $(y_n)_{n\in\mathbb{N}}$ produced by algorithms LCR or LCM belongs, when it exists. More generally, the question is: where are the accumulation points (limits of subsequences of $(y_n)_{n\in\mathbb{N}}$) located ? We introduce the set of stationary points for the general ODE:
\begin{equation*}
    \statio \coloneqq \{y \in \Rb^m, F(y)=0\},
\end{equation*}
 and the points that cancel $\dot{V}$
\begin{equation*}
    \zerodot \coloneqq \{y \in \Rb^m, \dot{V}(y)=0\},
\end{equation*}
for which we have $\statio\subset\zerodot$.
Besides, the set of critical points of $\R$ is the set 
\begin{equation*}
    \critic_\R \coloneqq \{\theta \in \Rb^N, \nR(\theta)=0\}.
\end{equation*}
Finally, in the discrete setting let us introduce the set of accumulation points of a sequence $(y_n)_{n\in \mathbb{N}}$: 
\begin{equation*}
    \accum \coloneqq \displaystyle{\bigcap_{p\in \mathbb{N}}} \overline{\{y_n, n\geq p\}}.
\end{equation*}
Depending on the optimizer, it is not always obvious that the accumulation points of $(y_n)_{n\in\mathbb{N}}$ intersects with $\statio$ or $\zerodot$. The object of this section is to study the inclusions of the different sets for the optimizers LCR and LCM.\\

Let us present some difficulties encountered when studying the accumulation points, first in the particular case of GD. In continuous time, the Lasalle invariance's principle \cite{LaSalle}, applied with $V=\R$, gives in particular that for each initial value $\theta_0 \in \Rb^N$, if $\theta(t)$ converges as $t$ goes to infinity then $\displaystyle{\lim_{t \to +\infty}} \theta(t) \in \zerodot=\statio$. In the discrete setting, when the time step is constant, the same property holds, that is to say, if $(\theta_n)_{n\in\mathbb{N}}$ converges then $\displaystyle{\lim_{n\to +\infty}}\theta_n \in \statio$: indeed,  taking the limit in \eqref{Euler} with $\eta_n=\eta>0$ leads to $\eta \nR(\theta_\infty)=0$, hence $\nR(\theta_\infty)=0$ {so that $\theta_{\infty} \in \critic_\R = \statio$}. In the general case where $(\eta_n)_{n\in\mathbb{N}}$ is not constant, this is much less straightforward. In the same way, we get $\eta_n \nR(\theta_n) \to 0$ but it is possible that $\eta_n \to 0$ and $\nR(\theta_n) \nrightarrow 0$. In \cite{Rondepierre}, this problem is solved, for GD, by assuming the gradient is globally Lipschitz for LCM and is only continuously differentiable for LCR. In the next lines, we begin by generalizing this result for a locally Lipschitz gradient for LCM.\\

\begin{proposition}[GD limit set]
    Let $\R$ be differentiable and assume its gradient is locally Lipschitz. Consider the sequence $(\theta_n)_{n\in\mathbb{N}}$ generated by the algorithm LCM with $F=\nR$ and $V=\R$ and assume that $(\theta_n)_{n\in\mathbb{N}}$ is bounded. Then the set of accumulation points $\accum$ of the sequence $(\theta_n)_{n\in\mathbb{N}}$ (limits of subsequences) is included in $\mathcal{E}=\critic_\R$. 
    \label{LCEGD_accumulation}
\end{proposition}
\begin{proof}
    Let us consider the compact set: 
    \begin{equation*}
        K = \overline{\{ \theta_n, n \in \mathbb{N}\}}.
    \end{equation*}
    Take an accumulation point $\theta^*$ and consider a subsequence $\theta_{\phi(n)}$ that converges to $\theta^*$. Denote by $L_K$ the Lipschitz constant of $\nR$ on $conv(K)$ where $conv(K)$ denotes the convex hull of $K$. Remember that in finite dimension, the convex hull of a compact set is compact. We have this classical inequality:
    \begin{equation}
    \forall y_1,y_2 \in conv(K), \R(y_2)-\R(y_1) \leq \nR(y_1)^T (y_2-y_1) + \frac{L_K}{2}\|y_2-y_1\|^2.
    \label{convex_inequality}
    \end{equation}
    Indeed we can write:
    \begin{multline*}
        \R(y_2)=\R(y_1)+\int_{0}^{1}\left(\nR(y_1+t(y_2-y_1))-\nR(y_1)+\nR(y_1)\right)^T(y_2-y_1)dt, \\
        = \R(y_1)+\nR(y_1)^T(y_2-y_1) + \int_0^1 \left(\nR(y_1+t(y_2-y_1))-\nR(y_1)\right)^T(y_2-y_1)dt, \\
        \leq \R(y_1)+\nR(y_1)^T(y_2-y_1) + \int_0^1 \|\nR(y_1+t(y_2-y_1))-\nR(y_1)\|\|y_2-y_1\|dt, \\
        \leq \R(y_1)+\nR(y_1)^T(y_2-y_1) + \int_0^1 L_Kt\|y_2-y_1\|^2dt,
    \end{multline*}
    by using Cauchy-Scharwtz and the definition of Lipshitz continuity.
    Applying this inequality to $y_1=\theta_n$ and $y_2=\theta_{n+1}$ it comes:
    \begin{equation*}
        \R(\theta_{n+1})-\R(\theta_n) \leq -\eta_n\left(1-\frac{L_K\eta_n}{2}\right)\|\nR(\theta_n)\|^2.
    \end{equation*}
    Therefore for {$\eta_n \leq \eta^* \coloneqq \frac{2}{L_K}(1-\lambda)$ the inequality \eqref{dissipation_inequality} is satisfied.}\\
    Now, take a look at the time step update. The algorithm starts the first iteration with the time step $\eta_{init}$, and at the iteration $n\geq 1$ we begin with a time step $f_2 \eta_{n-1}$. We have two complementary cases that may occur:
    \begin{enumerate}
        \item We begin with a time step $f_2 \eta_{n-1}$ smaller than $\eta^*$. So the inequality \eqref{dissipation_inequality} is already satisfied and supplementary computations are not needed to escape the repeat loop. Therefore $\eta_n=f_2 \eta_{n-1}$.
        \item If  $f_2 \eta_{n-1} \geq \eta^*$, we will reduce $f_2 \eta_{n-1}$ by $f_1$ several times. In the worst case, the algorithm has not found any solution greater than $\eta^*$ and we have to divide it one more time by $f_1$ so that $\eta_n \geq \frac{\eta^*}{f_1}$.
    \end{enumerate}
    As a result, the loop finishes with a time step $\eta_n \geq \min(\tilde{\eta}_n,\frac{\eta^*}{f_1})$ where $\tilde{\eta}_0 = \eta_{init}$ and $\tilde{\eta}_n  = f_2 \eta_{n-1}$ if $n > 0$. 
    By induction we have for $n\geq 0$:
    \begin{equation*}
         \eta_n \geq \min\left(f_2^n \eta_{init},\frac{\eta^*}{f_1}\right).
    \end{equation*}
    As $f_2>1$ there exists $n_1 \geq 0$ such that $ \forall n\geq n_1, f_2^n \eta_{init} \geq \frac{\eta^*}{f_1}$. Therefore:
    \begin{equation*}
        \forall n \geq 0, \eta_n \geq \min\left( \min_{0 \leq k < n_1}f_2^k \eta_{init}, \frac{\eta^*}{f_1}\right).
    \end{equation*}
    We can finally write the following inequality:
    \begin{equation*}
        \inf \eta_n \geq \min\left(\min_{0 \leq k < n_1}f_2^k \eta_{init}, \frac{\eta^*}{f_1}\right) > 0.
    \end{equation*}
    Assume by contradiction that $\theta^* \notin \statio$. Let us write the fundamental descent inequality for the subsequence $(\theta_{\phi_n})_{n\in\mathbb{N}}$:
    \begin{equation*}
        \R(\theta_{\phi(n+1)})-\R(\theta_{\phi(n)}) \leq -\lambda \eta_{\phi(n)} \|\nR(\theta_{\phi(n)})\|^2 \leq 0. 
    \end{equation*}
    So the sequence $(\R(\theta_{\phi(n)}))_{n\in\mathbb{N}}$ is a decreasing sequence bounded by below by 0 and therefore it converges. Then we can deduce that $\displaystyle{\lim_{n\to +\infty}} \eta_{\phi(n)} \|\nR(\theta_{\phi(n)})\|^2 = 0$. As $\inf \eta_{\phi(n)}>0$ and by the continuity of $\nR$ we deduce $\theta^* \in \statio$ which is a contradiction.\\
    Therefore we have proved that $\theta^* \in \statio$ where $\theta^*$ is any accumulation point of $(\theta_n)_{n\in \mathbb{N}}$..
\end{proof}

In the general case, we cannot expect convergence to $\statio$ since the continuous LaSalle's principle \citep{LaSalle} gives that for each initial value $y_0 \in \Rb^m$, $\omega(y_0) \subset \mathcal{Z}$ where:
\begin{equation*}
    \omega(y_0) = \{y^* \in \Rb^m, \exists t_n \to +\infty \text{ such that } y(0)=y_0 \text{ and } y(t_n) \to y^*\},
\end{equation*}
is called the limit set in ODE theory. It is the continuous equivalent of the set of accumulation points $\accum$ for sequence $(y_n)_{n\in\mathbb{N}}$. We want to extend the inclusion $\omega(y_0) \subset \mathcal{Z}$ to the discrete case: $\accum \subset \mathcal{Z}$. The most natural approach would be to apply the convex inequality \eqref{convex_inequality} to $V$ instead of $\R$. This inequality leads to:
\begin{equation*}
    V(y_{n+1})-V(y_n) \leq \nabla V(y_n)^T(y_{n+1}-y_n)+\frac{L_K}{2}\|y_{n+1}-y_n\|^2 = \eta_n\left(\dot{V}(y_n)+\frac{L_K\eta_n}{2}\|F(y_n)\|^2\right).
\end{equation*}
It is expected to only have $\dot{V}$ in the right part of the inequality to obtain the same form as the inequality \eqref{dissipation_inequality}. If we have $\|F(y)\|^2 \leq -\dot{V}(y)$ for all $y\in \Rb^m$ ({it is an equality} for GD) the previous inequality becomes:
\begin{equation*}
    V(y_{n+1})-V(y_n) \leq \eta_n\left(1-\frac{L_K\eta_n}{2}\right)\dot{V}(y_n).
\end{equation*}
Therefore we can proceed as in proposition \ref{LCEGD_accumulation}. But in the general case there is no reason that this inequality holds (see examples \ref{ex_Momentum_point} and \ref{ex_RMSProp_point}). In \cite{control2}, the authors \textbf{assume the existence of a continuous policy on $\Rb^m$ for the time step that satisfies inequality \eqref{dissipation_inequality} to deduce that the limit point of the sequence lies in $\mathcal{Z}$}. More formally they assume that there exists a continuous map $\tilde{s}: \Rb^m \mapsto \Rb_+^*$ such that:
\begin{equation}
    V(y_n+\tilde{s}(y_n)F(y_n))-V(y_n) \leq \lambda \tilde{s}(y_n)\dot{V}(y_n).
    \label{s_tilde}
\end{equation}
Here we will prove the existence of such an application but only continuous on $\Rb^m \setminus \mathcal{Z}$. This is the central tool to solve the problem of the limit points which is used later to deal with convergence properties of the algorithms. 
Note that although it is an abstract result of existence (the continuous function involved in the theorem is not explicited in this paper),
it is sufficient to 
obtain several properties of
the optimizers.\\ 

\begin{theorem}[Selection Theorem]
    Assume that $V \in \mathcal{C}^2(\Rb^m)$. Then, there exists a continuous function $s:\Rb^m\setminus \mathcal{Z} \mapsto \Rb_+^*$ such that:
    \begin{equation*}
        \forall y \in \Rb^m\setminus \zerodot, \forall \eta \in ]0,s(y)]: f(y,\eta) \leq 0.
    \end{equation*}
    \label{selection_theorem}
\end{theorem}

The idea is to see the object $I$ defined in \eqref{I} as a multi-application (or set value map) 
$$
I:\left\{\begin{array}{l}
\Rb^m \rightarrow \mathcal{P}(\Rb_+^*),\\
y \rightarrow I(y),
\end{array}\right.
$$
where $\mathcal{P}(\Rb_+^*)$ denotes the set of subsets of $\Rb_+^*$. 
In other words, a multi-application matches a vector to a set. Under this point of view $\tilde{s}$ defined in \eqref{s_tilde} can be seen as a continuous selection of $I$, that is to say a continuous map satisfying:
\begin{equation*}
    \forall y \in \Rb^m \text{ , } \tilde{s}(y) \in I(y).
\end{equation*}
Here instead of assuming the existence of this map, we will prove the existence of a slightly weaker continuous selection $s$ on $\Rb^m \setminus \mathcal{Z}$:
\begin{equation*}
    \forall y \in \Rb^m\setminus \zerodot \text{ , } s(y) \in I(y).
\end{equation*}

The construction of such continuous map is known as the selection theory (see \cite{infinite_dimensional_analysis} for an introduction). Thanks to the multi-application point of view we are reducing the accumulation points localization problem to a topological problem well studied in the literature. 
Nevertheless, the vast majority of results (see \cite{infinite_dimensional_analysis, selection_book} and theorem \ref{theorem_used} of the appendix \ref{selection_theory} for the theorem applied in this section) assume that the \textbf{value maps are convex sets}. Unfortunately, in our case, there is no reason for $I(y)$ to be convex for all $y$ in $\Rb^m$. In order to apply one of these results, we have to find a convex subset $T(y)\subset I(y)$ and build a continuous selection $s$ restricted to $T$. The Taylor-Lagrange formula applied to $f$, defined in \eqref{f}, helps us find a natural candidate for $T(y)$ (see the proof of lemma \ref{inclusion_bounded}). Let us first introduce some notations:
\begin{itemize}
    \item For $y \in \Rb^m$ and $x \in \Rb_+$: $g(y,x) \coloneqq \left\lvert F(y)^T \nabla^2V(y+xF(y)) F(y) \right\rvert$.
    \item For $y \in \Rb^m$ and $\eta \in \Rb_+$
    \begin{equation*}
    q(y,\eta) \coloneqq \eta \left(\displaystyle{\max_{x \in [0,\eta]}} g(y,x)+1 \right) + 2(1-\lambda)\dot{V}(y).
    \end{equation*}
\begin{remarknot}
    The constant $1$ added to the max in the definition of $q$ may seem arbitrary. In fact it is possible to take any positive value. The role of this constant is to enforce the strict increasing monotony of $q$ (without this constant $q$ is just increasing) which constitutes an important feature in the proofs.
\end{remarknot}
    \item The multi-valued application is denoted by 
    \begin{equation*}
        T:
        \left\{
        \begin{array}{l}
            \Rb^m \setminus \zerodot \mapsto \mathcal{P}(\Rb_+^*) \\\\
            y \longmapsto \{\eta>0, q(y,\eta)\leq 0\}.
        \end{array}
        \right.
    \end{equation*}
\end{itemize}

Let us begin by proving the inclusion claimed underneath:
\begin{lemma}
    $\forall y \in \Rb^m \setminus \zerodot$, $T(y) \subset I(y)$ and $I(y)$ is bounded. 
    \label{inclusion_bounded}
\end{lemma}

\begin{proof}
    The inclusion comes from the Taylor-Lagrange formula as $f \in \mathcal{C}^2(\Rb^N)$:
    \begin{equation*}
        \left\lvert f(y,\eta)-f(y,0)-\eta \frac{\partial f}{\partial \eta}(y,0)\right\rvert \leq \frac{\eta^2}{2} \displaystyle{\max_{x\in [0,\eta]}} \left\lvert \frac{\partial^2 f}{\partial \eta^2}(y,x) \right\rvert.
    \end{equation*}
    And since:
    \begin{equation*}
        \begin{array}{lll}
            \frac{\partial f}{\partial \eta}(y,\eta) = F(y)^T \left[ -\lambda \nabla V(y)+\nabla V(y+\eta F(y))\right], \\\\
            \frac{\partial^2 f}{\partial \eta^2}(y,\eta) = F(y)^T \nabla^2 V(y+\eta F(y)) F(y).
        \end{array}
    \end{equation*}
    This implies:
    \begin{equation*}
         \left\lvert f(y,\eta)+\eta (\lambda-1)\dot{V}(y)\right\rvert \leq \frac{\eta^2}{2} \displaystyle{\max_{x\in [0,\eta]}} g(y,x).
    \end{equation*}
    Let $\eta \in T(y)$. We have by definition:
    \begin{equation*}
        \eta \left(\displaystyle{\max_{x \in [0,\eta]}} g(y,x)+1 \right) \leq -2(1-\lambda)\dot{V}(y).
    \end{equation*}
    Then:
       \begin{equation*}
        \eta \displaystyle{\max_{x \in [0,\eta]}} g(y,x) \leq -2(1-\lambda)\dot{V}(y).
    \end{equation*}
    Therefore:
    \begin{multline*}
        f(y,\eta)-\eta (1-\lambda)\dot{V}(y) \leq \left\lvert f(y,\eta)-\eta (1-\lambda)\dot{V}(y)\right\rvert \leq \frac{\eta}{2}\left(\eta \displaystyle{\max_{x \in [0,\eta]}} g(y,x) \right) \\ 
        \leq -\eta(1-\lambda)\dot{V}(y).
    \end{multline*}
    Then $f(y,\eta)\leq 0$ which gives the first part of the lemma.
    \\    
    For the second part, by contradiction assume that $I(y)$ is not bounded. Then we can build a sequence $\eta_n \to +\infty$ such that for all $n\geq 0$: $f(y,\eta_n)\leq 0$. This leads to $\displaystyle{\lim_{n\to +\infty}} \lambda \eta_n \dot{V}(y) = -\infty$ and the inequality:
    \begin{equation*}
        V(y+\eta_nF(y))-V(y) \leq \lambda \eta_n \dot{V}(y)
    \end{equation*}
    gives $\displaystyle{\lim_{n\to +\infty}} V(y+\eta_n F(y)) = -\infty$. This is in contradiction with the positivity of $V$.
\end{proof}

By building the set value map $T$, we have enforced the first condition of theorem \ref{theorem_used} about convex value maps. The second central condition of this theorem (and more generally in selection theory) is \textbf{the lower hemicontinuity} recalled in the appendix \ref{selection_theory}. This is closely related to the existence of a local continuous solution $\eta$ (continuous as a function of $y$) to the equation $q(y,\eta)=v$ for some fixed value $v$. That is why we have to prove a lemma which can be seen as an implicit function theorem: it is very close to the implicit function theorem for strictly monotone functions stated in \citet{implicit_theorem} p.63, but the authors require the continuity respect to the couple $(x,y)$.\\

\begin{lemma}[Increasing implicit function lemma]
    Consider a function $q:\mathcal{O} \times \Rb_+^* \mapsto \Rb$ where $\mathcal{O}$ is an open subset of $\Rb^m$ and such that:
    \begin{enumerate}
        \item For all $y\in \Rb_+^*$, $x \mapsto q(x,y)$ is continuous on $\mathcal{O}$.
        \item For all $x\in \mathcal{O}$, $y \mapsto q(x,y)$ is continuous and strictly increasing on $\Rb_+^*$.
    \end{enumerate}
    Consider $(a,b) \in \mathcal{O} \times \Rb_+^*$ such that $q(a,b)=0$. Then there exists a neighborhood $\mathcal{V}$ of $a$ and a continuous map $\phi: \mathcal{V} \mapsto \Rb_+^*$ such that $q(x,y)=0 \Leftrightarrow \forall x\in \mathcal{V}, y=\phi(x)$.
    \label{increasing_implicit}
\end{lemma}
\begin{proof}
    Consider $r>0$ such that $b-r>0$. 
    Since $y \mapsto q(a,y)$ is strictly increasing on $\Rb_+^*$, we have:
    \begin{equation*}
        q(a,b-r)<0 \text{ and } q(a,b+r)>0.
    \end{equation*}
    Moreover, the continuity of $q$ with respect to $x$ gives the existence of $\alpha>0$ satisfying:
    \begin{equation*}
        \forall x \in B(a,\alpha): q(x,b-r)<0 \text{ and } q(x,b+r)>0. 
    \end{equation*}
    Indeed, assume by contradiction that such an $\alpha$ does not exist:
    \begin{equation*}
        \forall \alpha>0, \exists x \in B(a,\alpha), q(x,b-r)\geq 0 \text{ or } q(x,b+r) \leq 0. 
    \end{equation*}
    Taking the sequence $\alpha_n = \frac{1}{n}>0$ for $n\geq 1$, the property above makes it possible to build a sequence $x_n \in B\left(a,\frac{1}{n}\right)$ such that:
    \begin{equation*}
        q(x_n,b-r)\geq 0 \text{ or } q(x_n,b+r)\leq0.
    \end{equation*}
    As $\|x_n-a\|< \frac{1}{n}$ for all $n\geq 1$ we deduce that $x_n$ converges to $a$. By continuity of $x\mapsto q(x,b\pm r)$ and by passing to the limit in the two inequalities above we have: $q(a,b-r)\geq 0$ or $q(a,b+r)\leq0$ which is a contradiction.\\   
    For each $x$ in this ball $\mathcal{V}:=B(a,\alpha)$, we can find $y_0(x) \in ]b-r,b+r[$ satisfying $q(x,y_0(x))=0$ by the intermediate value theorem. It is unique since $y \mapsto q(x,y)$ is a one-to-one map. Let us denote by $\phi(x)$ this number $y_0(x)$ for each $x\in \mathcal{V}$, it remains to prove the continuity of $\phi$. \\
    Let $x_0 \in \mathcal{V}$ and show the continuity in $x_0$. We can write $q(x_0,y_0)=0$ where $y_0=\phi(x_0)$. Let $\epsilon>0$. Once again invoking the fact that $y \mapsto  q(x_0,y)$ is strictly increasing on $\Rb_+^*$ we get $q(x_0,y_0-\epsilon)<0$ and $q(x_0,y_0+\epsilon)>0$. The continuity 
    of $q$ with respect to its first variable gives the existence of $\gamma>0$ satisfying:
    \begin{equation*}
        \forall x \in B(x_0,\gamma): q(x,y_0-\epsilon)<0 \text{ and } q(x,y_0+\epsilon)>0. 
    \end{equation*}
    The intermediate value theorem gives that $\phi(x) \in ]y_0-\epsilon, y_0+\epsilon[$, which concludes this proof.
\end{proof}

Now we have to check that the conditions above are verified by our application $q$ (the constant $1$ enables to have a strictly increasing function instead of a increasing function).
Lemma \ref{increasing_implicit} will be applied to a translation of $q$ in the proof of lemma \ref{lem_map_T} because the lower hemicontinuity of $T$ is closely related to the existence of a continuous solution $\phi(x)$ of the inequality $q(x,\phi(x))\leq 0$.\\

\begin{lemma}
    We have:
    \begin{enumerate}
        \item $\forall \eta \in \Rb_+$, $y \mapsto q(y,\eta)$ is continuous.
        \item $\forall y\in \Rb^m$, $\eta \mapsto q(y,\eta)$ is continuous and strictly increasing.
    \end{enumerate}
    \label{continuity_increasing}
\end{lemma}
\begin{proof}
    Let us prove the second point first. Consider a fixed $y \in \Rb^m$. To prove the continuity of $\eta \mapsto q(y,\eta)$, the challenging part is 
    to prove the continuity of $\eta \mapsto \displaystyle{\max_{x \in [0,\eta]}} g(y,x)$. Let $\eta \in \Rb^+$ and $\epsilon>0$. As $g(y,\cdot)$ is continuous at $\eta$, there exists $\gamma>0$ (depending of $\eta$ and $y$) such that:
    \begin{equation*}
        \forall \eta' \in \Rb^+, |\eta-\eta'|<\gamma \Rightarrow |g(y,\eta)-g(y,\eta')|<\epsilon.
    \end{equation*}
    So for such a $\eta'$ we have: $g(y,\eta)-\epsilon < g(y,\eta')<g(y,\eta)+\epsilon$. For $\eta'$ satisfying $0<\eta'-\eta<\gamma$ we can deduce:
    \begin{equation*}
        \displaystyle{\max_{x \in [0,\eta']}} g(y,x) < \displaystyle{\max_{x \in [0,\eta]}} g(y,x)+\epsilon.
    \end{equation*}
    By symmetry of the role of $\eta$ and $\eta'$ we conclude about the continuity with respect to $\eta$. Concerning the monotonicity, it is sufficient to notice that $\eta \mapsto q(y,\eta)$ is the sum of the increasing function $\eta \mapsto \displaystyle{\max_{x \in [0,\eta]}} g(y,x)$, the strict increasing (linear) function $\eta \mapsto \eta$ and a rest which is independent on $\eta$.
    
    For the first point, let us consider $\eta \in \Rb_+$ and show the continuity of $q$ at $y \in \Rb^m$. Let $\epsilon>0$. As the function $g$ is continuous it is uniformly continuous on the compact $\bar{B}(y,1) \times [0,\eta]$. Then there exists $\gamma>0$ (depending of $\eta$ and $y$) such that:
    \begin{equation}
        \forall y' \in \bar{B}(y,1), \forall x,x' \in [0,\eta], \|y-y'\|_{\infty}<\gamma \text{ and } |x-x'|<\gamma \Rightarrow |g(y,x)-g(y',x')|<\epsilon.
        \label{uniform_continuity}
    \end{equation}
    Now consider a tessellation of non-overlapping sets of the compact $[0,\eta]$, $x_0=0$, $x_1$, $\dots$, $x_n=\eta$ such that for all $0\leq i \leq n$, $|x_{i+1}-x_i|<\gamma$ (all $x_i$ depend of $\eta$ and $\gamma$). We can then write the following equality:
    \begin{equation*}
        \displaystyle{\max_{x \in [0,\eta]}} g(y',x) = \displaystyle{\max_{0\leq i \leq n}} \left( \displaystyle{\max_{x' \in [x_i,x_{i+1}]}} g(y',x')\right).
    \end{equation*}
    Let $y'$ be such that $\|y-y'\|_{\infty}<\gamma$. By \eqref{uniform_continuity} and by passing to the maximum we claim that for all $0 \leq i \leq n$:
    \begin{equation*}
        \displaystyle{\max_{x' \in [x_i,x_{i+1}]}} g(y,x') -\epsilon \leq \displaystyle{\max_{x' \in [x_i,x_{i+1}]}} g(y',x') \leq \displaystyle{\max_{x' \in [x_i,x_{i+1}]}} g(y,x')+\epsilon. 
    \end{equation*}
    To prove this, assume that it is not the case. By continuity of $x' \mapsto g(y,x')$ on the compact set $[x_i,x_{i+1}]$ it exists $x_1',x_2'\in [x_i,x_{i+1}]$ satisfying $g(y,x_1')=\displaystyle{\max_{x' \in [x_i,x_{i+1}]}} g(y,x')$ and $g(y',x_2')=\displaystyle{\max_{x' \in [x_i,x_{i+1}]}} g(y',x')$. Therefore $|g(y,x_1')-g(y',x_2')|>\epsilon$ which is in contradiction with \eqref{uniform_continuity}.  
    By taking the maximum with respect to the finite number of indices $i$, we have the continuity with respect to $y$.\\
\end{proof}

Now let us verify that $T$ has the nice properties claimed before, mainly convexity
values and lower hemicontinuity \ref{hemicontinuity}.\\

\begin{lemma}
    The map $T$ has non-empty closed convex values in $\Rb_+^*$ and it is lower hemicontinuous.
    \label{lem_map_T}
\end{lemma}
\begin{proof}
    Let $y \in \Rb^m \setminus \zerodot$. The fact that $\eta \mapsto q(y,\eta)$ increases gives that $T(y)$ is an interval. \\
    $T(y) = \Rb_+^* \cap \{\eta \geq 0, q(y,\eta)\leq 0\}$ is closed in $\Rb_+^*$ because of the continuity of $\eta \mapsto q(y,\eta)$. \\
    Assume by contradiction that $T(y)$ is empty. This means that:
    \begin{equation*}
        \forall \eta>0, \eta \left(\displaystyle{\max_{x \in [0,\eta]}} g(y,x)+1
        \right) > -2(1-\lambda)\dot{V}(y).
    \end{equation*}
    Let $\eta$ tends towards 0 and we get that $\dot{V}(y)=0$. This is a contradiction because $y \notin \zerodot$. Thereofore $T(y)$ is not empty. \\

    Let $y \notin \zerodot$ and show the lower continuity of $T$ in $y$. Let $U$ be an open set that intersects $T(y)$. Lemma \ref{inclusion_bounded} states that $T(y)$ is a non-empty bounded interval, hence we can write it as:
    \begin{equation*}
        T(y) = (a(y),b(y)),
    \end{equation*}
    where $a(y)$ and $b(y)$ can be included or not. 
    As $U$ intersects $T(y)$ and $U$ is the union of open intervals, there exists $\epsilon>0$ and $\eta \in T(y)$ such that: $]\eta-\epsilon,\eta+\epsilon[ \subset U$.\\
    
    Define $v=q(y,\eta) \leq 0$ and the function $\tilde{q}(x,y) = q(x,y)-v$. We apply the increasing implicit theorem (lemma \ref{increasing_implicit}) on $\tilde{q}$ and $(y,\eta)$ to get the existence of $r>0$ and a continuous map $\phi: B(y,r) \mapsto \Rb$ such that:
    \begin{equation*}
        \forall x \in B(y,r), \tilde{q}(x,y)=0 \Leftrightarrow y=\phi(x).
    \end{equation*}
    By the continuity of $\phi$ {with respect to} $y$ and the fact that $\phi(y)=\eta>0$ 
    there exists $\alpha>0$ such that:
    \begin{equation*}
        \forall x \in B(y,\alpha), \phi(x)>0 \text{ and } |\phi(x)-\phi(y)| < \epsilon.
    \end{equation*}
    Let $x \in B(y, \alpha)$. We claim 
    that $U$ intersects $T(x)$.
    Indeed we have $\tilde{q}(x,\phi(x))=0$ which means that $q(x,\phi(x)) = v \leq 0$. Therefore $\phi(x) \in T(x)$. But $\phi(x) \in ]\eta-\epsilon, \eta+\epsilon[ \subset U$. So $\phi(x) \in T(x) \cap U$.\\
\end{proof}

Equipped with these results, it is now possible to prove the selection theorem (\ref{selection_theorem}) by seeing it as
an application of theorem \ref{theorem_used} recalled in appendix \ref{selection_theory}.
\begin{proof}{\textbf{of theorem \ref{selection_theorem} (Selection Theorem)}}
    Using the previous lemma \ref{lem_map_T}, we apply the selection theorem 6.2 p.116 in \cite{selection_book}. This gives a continuous application $s:\Rb^m \setminus \zerodot \mapsto \Rb_+^*$ such that:
    \begin{equation*}
        \forall \theta \in \Rb^m\setminus \mathcal{Z}, q(y,s(y)) \leq 0.
    \end{equation*}
    As $T(y)$ is an interval and 0 is its infimum we have:
       \begin{equation*}
        \forall y \in \Rb^m\setminus \zerodot, \forall \eta \in ]0,s(y)], q(y,\eta) \leq 0.
    \end{equation*}
    The result is a direct consequence of the inclusion $T(y) \subset I(y)$ coming from the lemma \ref{inclusion_bounded}.
\end{proof}

Now we can prove that the set of accumulation points for LCR lies in $\zerodot$: from the previous result, we can replace $\eta^*$ in the proof of proposition \ref{LCEGD_accumulation} by the minimum of $s$ on some compact. {\it To the best of our knowledge, it is the first time selection theory is applied to backtracking optimizers.} The following result can be seen as a discrete LaSalle principle:\\

\begin{theorem}[LCR limit set]
    Assume that $V \in \mathcal{C}^2(\Rb^m)$. Consider the sequence $(y_n)_{n\in\mathbb{N}}$ generated by the algorithm LCR and assume that $(y_n)_{n\in\mathbb{N}}$ is bounded. Then the set of accumulation points of the sequence $(y_n)_{n\in\mathbb{N}}$ (limits of subsequences) lies in $\zerodot$, i.e. $\accum\subset\zerodot$. 
    \label{LCR_accumulation}
\end{theorem}

\begin{proof}
    Consider a subsequence $(y_{\phi(n)})_{n\in\mathbb{N}}$ that converges to $y^* \in \Rb^m$. Passing to the limit in the relation $V(y_{\phi(n+1)})-V(y_{\phi(n)}) \leq \lambda \eta_{\phi(n)} \dot{V}(y_{\phi(n)}) \leq 0$, we get that $\displaystyle{\lim_{n\to +\infty} \eta_{\phi(n)} \dot{V}(y_{\phi(n)}) = 0}$. Indeed, note that the sequence $(V(y_{\phi(n)}))_{n\in\mathbb{N}}$ converges since it is decreasing and lower bounded by 0. \\ 
    
 Assume by contradiction that $y^* \notin \zerodot$. As $(y_{\phi(n)})_{n\in\mathbb{N}}$ converges to $y^*$ and $\zerodot$ is closed (due to the continuity of $\dot{V}$), there exists a compact set $K$ containing $y^*$ and no points of $\zerodot$, such that $\forall n \geq n_0, y_{\phi(n)} \in K$ for a certain $n_0\geq 0$. We consider the function $s$ of the selection theorem \ref{selection_theorem}. As $s$ is continuous on $K\subset \Rb^N\setminus \zerodot$, we define $\eta^* = \displaystyle{\min_{y \in K}} \text{ } s(y)>0$. By the property on $s$ we deduce:
\begin{equation*}
    \forall y \in K, \forall \eta \in ]0,\eta^*], V(y+\eta F(y)) - V(y) \leq \lambda \eta \dot{V}(y).
\end{equation*}

At the iteration $\phi(n) \geq n_0$ we begin with a time step $\eta_{init}$. The loop finishes with a time step $\eta_{\phi(n)} \geq \min(\eta_{init},\frac{\eta^*}{f_1})$ for the same reason than in the proof of proposition \ref{LCEGD_accumulation}. Then $\inf \eta_{\phi(n)} \geq \min\left(\displaystyle{\min_{0\leq k< n_0}\eta_k},\eta_{init},\frac{\eta^*}{f_1}\right)>0$.\\

As $(y_{\phi(n)})_{n\in\mathbb{N}}$ converges and $\dot{V}$ is continuous, $(\dot{V}(y_{\phi(n)}))_{n\in\mathbb{N}}$ converges to $\dot{V}(y^*)$. Now, since $\displaystyle{\lim_{n\to +\infty}} \eta_{\phi(n)} \dot{V}(y_{\phi(n)}) = 0$ and $\inf \eta_{\phi(n)}>0$ we can conclude that $\dot{V}(y^*)=0$. This is a contradiction so $y^* \in \zerodot$.\\
\end{proof}

Unfortunately, for LCM, the existence of a continuous selection on $\Rb^m\setminus \zerodot$ is not sufficient to locate all the accumulation points of the sequence $(y_n){_{n\in\mathbb{N}}}$ generated by LCM. We next prove a weaker version of {the previous} result. Still, note that the next theorem remains sufficient to deduce the convergence result of section \ref{section_cv}.\\

\begin{theorem}[LCM limit point]
    Assume that $V \in \mathcal{C}^2(\Rb^m)$. Consider the sequence $(y_n)_{n\in\mathbb{N}}$ generated by the algorithm LCM and assume that $(y_n)_{n\in\mathbb{N}}$ is bounded. Then there exists at least one accumulation point of $(y_n)_{n\in\mathbb{N}}$ in $\zerodot$.  
    \label{LCM_accumulation}
\end{theorem}
\begin{proof}
     By contradiction assume that there is no accumulation point in $\zerodot$.
     Let us define the following set:
     \begin{equation*}
        K = \overline{\{ y_n, n \in \mathbb{N}\}}\ \setminus \zerodot \subset \Rb^m\setminus \zerodot.
    \end{equation*}
    We claim that $K$ is compact. First of all, $K$ is bounded since $(y_n)_{n\in\mathbb{N}}$ is bounded. To show that $K$ is closed, let us take a convergent sequence of elements of $K$, that is to say, a subsequence $(y_{\phi(n)})_{n\in\mathbb{N}}$ of $(y_n)_{n\in\mathbb{N}}$ such that: $\forall n\in \mathbb{N}, y_{\phi(n)} \notin \zerodot$. If $\displaystyle{\lim_{n\to \infty}}y_{\phi(n)} \in \zerodot$, we have found an accumulation point lying in $\zerodot$, which is a contradiction. Therefore $\displaystyle{\lim_{n\to \infty}}y_{\phi(n)} \in K$.\\
    To deduce the result, the methodology used 
    in the proof of proposition \ref{LCEGD_accumulation} can be used
    just by replacing $\eta^*\coloneqq \frac{2}{L_K}(1-\lambda)$ by $\eta^* \coloneqq \displaystyle{\min_{y\in K}} \text{ } s(y)$.
\end{proof}

\begin{remark}
    This result implies that if $(y_n)_{n\in\mathbb{N}}$ converges, its limit lies in $\zerodot$. In order to get a result as general as in theorem \ref{LCR_accumulation} but for LCM, we may have to build a selection $s$ continuous on $\Rb^m$ {instead of $\Rb^m\setminus \zerodot$}.
    \label{remark_LCM_zerodot}
\end{remark}

\begin{remarknot}
    Both theorems \ref{LCR_accumulation} and \ref{LCM_accumulation} assume the existence of bounded trajectories. A sufficient condition to ensure that the sequence $(y_n)_{n\in \mathbb{N}}$ is bounded for any initial condition $y_0 \in \Rb^m$ is to assume that $V$ is radially unbounded, that is to say: $\displaystyle{\lim_{\|y\|\to \infty}} V(y)=+\infty$. Indeed, if the sequence is not bounded, it is possible to build a subsequence $\|y_{\phi(n)}\| \to +\infty$. This implies that $V(y_{\phi(n)}) \to +\infty$. But this is a non sense because the sequence $(V(y_n))_{n\in \mathbb{N}}$ is decreasing due to condition \eqref{dissipation_inequality}.
\end{remarknot}

To end this section, we present some consequences of our results on the classical optimizers already mentioned in section \ref{intro}. \\

\begin{example}[GD]
    Let $\R$ be differentiable with $\nR$ locally Lipshitz and $\R$ radially unbounded. For GD, {$\zerodot=\statio=\critic_\R$}. Then proposition \ref{LCEGD_accumulation} ensures that the accumulation points of GD based on LCM are critical points of the function to minimize, as $\accum\subset \critic_\R$. This is a desired property for an optimizer. If $\R$ is not radially unbounded, it is possible to add, for example, an $L^2$-regularization term for the result to hold.
    \label{ex_GD_point}
\end{example}

\begin{example}[Momentum]
    Let $\R \in \mathcal{C}^2(\Rb^N)$. For Momentum, let us remember the expression of $F$ ($y=(v,\theta)^T$):
    \begin{equation*}
        F(v,\theta) = 
        \begin{pmatrix}
            -v-\nR(\theta) \\
            v
        \end{pmatrix},
    \end{equation*}
    where the constant parameter $\bar{\beta_1}$ is set to one for the sake of simplicity.
    We can easily see that the function $V(v,\theta) = \R(\theta)+\frac{\|v\|^2}{2}$ is positive, radially unbounded, and by computing its derivative $\dot{V}(v,\theta) = -\|v\|^2$, we get that $\dot{V}(v,\theta)\leq 0$. 
    By applying the algorithm LCR with this Lyapunov control $V$, we have that all subsequences satisfy $v_{\phi(n)} \to 0$ using theorem \ref{LCR_accumulation}. But we get no valuable information about {the limit of}  $\nR(\theta_{\phi(n)})$ because $\zerodot \neq \statio$ {as} $\zerodot = \{(0,\theta),\theta \in \Rb^N\}$ and $\statio=\{(0,\theta), \theta \in \critic_\R\}$.
    \label{ex_Momentum_point}
\end{example}

\begin{example}[RMSProp]
    Let $\R \in \mathcal{C}^2(\Rb^N)$ and consider the RMSProp ODE \citep{CV_Adam, Bilel} that has the form $y=(s,\theta)^T\in \Rb_+^N \times \Rb^N$ and:
    \begin{equation*}
        F(s,\theta) =
        \begin{pmatrix}
            -s + \nR(\theta)^2 \\
            - \dfrac{\nR(\theta)}{\sqrt{\epsilon_a+s}}
        \end{pmatrix}
        .
    \end{equation*}
    Consider algorithms LCR and LCM with the following Lyapunov function:
    \begin{equation*}
        V(s,\theta) = 2\left( \R(\theta) + \displaystyle{\sum_{i=1}^N \sqrt{\epsilon_a+s_i}}\right),
    \end{equation*}
    and its derivative:
    \begin{equation*}
        \dot{V}(s,\theta) = -\displaystyle{\sum_{i=1}^N} \dfrac{s_i}{\sqrt{\epsilon_a+s_i}}-\displaystyle{\sum_{i=1}^N} \dfrac{\partial_i\R(\theta)^2}{\sqrt{\epsilon_a+s_i}} \leq 0.
    \end{equation*}
    Note that the equivalence $\dot{V}(s,\theta)=0 \Leftrightarrow \left\{s=0 \text{ and } \nR(\theta)=0\right\}$ holds so that $\statio=\zerodot$. {Furthermore,} $V$ is positive radially unbounded. Hence for LCR the accumulation points $(s^*,\theta^*)\in\accum$ satisfy $s^*=0$ and $\nR(\theta^*)=0$. These equalities are true for the potential limit $(s^*, \theta^*)$ of LCM. Therefore if RMSProp with the control $V$ converges, the output of the algorithms will be a critical point of $\R$ i.e. $\mathcal{A}\subset \{0\} \times \critic_\R$.  
    \label{ex_RMSProp_point}
\end{example}

\begin{example}[pGD]
     Let $\R \in \mathcal{C}^2(\Rb^N)$ and consider the pGD flow \cite{cv_qflow} (with $p>1$) that has the form $y=\theta\in \Rb^{N}$ and:
    \begin{equation*}
        F(\theta) =
        \left\{
        \begin{array}{cll}
             \dfrac{\nR(\theta)}{\|\nR(\theta)\|^{\frac{p-2}{p-1}}} &\text{ if } &\nR(\theta)\neq 0, \\\\
             0 &\text{  if } &\nR(\theta)=0. 
        \end{array}
        \right.
    \end{equation*}
    As $p>1$, $\dfrac{p-2}{p-1}<1$ so $F$ is continuous. We can take the Lyapunov function of GD: $V(\theta)=\R(\theta)$. Few computations give:
    \begin{equation*}
        \dot{V}(\theta) = -\|\nR(\theta)\|^{\frac{p}{p-1}} \leq 0.
    \end{equation*}
    For this optimizer, by definition $\critic_\R = \statio$. Besides, given the expression of $\dot{V}$, we have $\zerodot=\statio$. As a result, if the algorithms LCR and LCM converge, the result is a critical point of the function to minimize, i.e. we have $\accum\subset \zerodot=\critic_\R$. In the particular case $1<p\leq 2$, the hypothesis $\R \in \mathcal{C}^2$ can be weakened by just requiring that $\nR$ is locally Lipshitz. Indeed, following the same argument than in proposition \ref{LCEGD_accumulation}, we can apply the convex inequality \eqref{convex_inequality} to $V$:
    \begin{equation*}
        V(\theta_{n+1})-V(\theta_n) \leq \eta_n\left(\dot{V}(\theta_n)+\frac{L_K\eta_n}{2}\|F(\theta_n)\|^2\right).
    \end{equation*}
    But $\|F(\theta)\|^2 = (-\dot{V}(\theta))^{\frac{2}{p}} \leq -\dot{V}(\theta)$ for points $\theta$ in a neighborhood of $\theta^* \in \mathcal{E}$ because $p\leq 2$. The above inequality leads to:
    \begin{equation*}
        V(\theta_{n+1})-V(\theta_n) \leq \eta_n\left(1-\frac{L_K\eta_n}{2}\right)\dot{V}(\theta_n).
    \end{equation*}
    We conclude as in the proof of lemma \ref{LCEGD_accumulation}. However, we will see in example \ref{pgd_cv} that the case $p\leq 2$ is of little interest.
    \label{ex_pGD_point}
\end{example}

\section{The local asymptotic stability}
\label{section_stability}

In the previous section, we have located the accumulation points of the sequences generated by algorithms LCR, LCM, and illustrated on several examples that the induced optimizers lead to critical points of the function that we would like to minimize. One may argue that this requirement is trivially satisfied by constant step size optimizers and that the suggested Lyapunov backtrackings only bring technical difficulties. We will {\it prove that this is not the case by digging out some crucial qualitative properties of the {backtracking algorithms LCR and LCM which, to our knowledge, do not hold in the same conditions for their constant step size counterparts}}. {More precisely, we prove stability in this section and global convergence in the next one}.

\subsection{Local stability of an isolated stationary point}
\label{local_stability}

Recently in the analysis of optimizers for ML, many papers got interested in the stability properties of an ODE $y'(t)=F(y(t))$ which is asymptotically solved (i.e. as $\eta\rightarrow 0$) by the iterates of the studied optimizer. For instance, GD asymptotically solves the ODE $\theta'(t) = -\nR(\theta(t))$ for which it is well-known \citep{Bilel} that the isolated minimums of $\R$ are stable. In \cite{Bilel}, the importance of preserving this property after discretization is illustrated and discussed on numerical examples. 
In \cite{Absil}, the authors prove this property only for descent algorithms like GD, under Armijo's conditions for the time step selection, and for analytic functions. Here we generalize this result for a general ODE/optimizer and suppress the analytical assumption. 
This means that the backtracking of algorithm LCR (the case of LCM is tackled later) is sufficient for the discrete process to preserve local stability behavior of the ODE. 
{
}\\

Let us claim the main result of this section that allows checking the stability of algorithm LCR. This theorem is a sort of discrete Lyapunov theorem \cite{Lyapunov}.\\
\begin{theorem}[Stability Theorem]
    Consider an equilibrium $y^*$ of the ODE $y'(t)=F(y(t))$ i.e. $y^*\in \statio$. Assume that the Lyapunov function $V \in \mathcal{D}(\Rb^m)$ (differentiable functions on $\Rb^m$) in the algorithm LCR is definite positive and that its derivative $\dot{V}$ is negative on some neighborhood $B_r(y^*)$ of $y^*$ ($r>0$) and that $V(y^*)=0$. Then $y^*$ is a stable equilibrium of the algorithm LCR:
    \begin{equation*}
        \forall \epsilon>0, \exists \gamma>0, \|y_0-y^*\|<\gamma \implies \forall n\geq 0, \|y_n-y^*\|<\epsilon.
    \end{equation*}
    \label{stability_theorem}
\end{theorem}

\begin{proof}
    Before beginning the proof,let us rewrite the LCR algorithm as a recurrent relation $y_{n+1}=h(y_n)$ where:
    \begin{equation*}
        h(y) = y+\tilde{\eta}(y)F(y).
        \label{h}
    \end{equation*}
    with:
    \begin{equation*}
        \left\{
        \begin{array}{ll}
        \tilde{\eta}(y) = \dfrac{\eta_0}{f_1^p}, \\\\
         p = \min \left\{k\in \mathbb{N}, f\left(y,\dfrac{\eta_0}{f_1^k}\right)\leq 0\right\}.
        \end{array}
        \right.
    \end{equation*}
    The main difficulty is the \textbf{non-continuity} of $h$. In fact, classical analysis of discrete stability assume the continuity of the application that generates the sequence \cite{LaSalle_discrete, Lyapunov_discrete, Lyapunov_inria, Lyapunov_non_monotonic}. To overcome this difficulty, we will not focus directly on the trajectory starting from a close initial condition but on a sequence of initial conditions. Assume by contradiction that $y^*$ is not a stable equilibrium:
    \begin{equation*}
        \exists \epsilon>0, \forall \gamma>0, \exists \tilde{y}_0 \in \Rb^N, \exists m \in \mathbb{N}, \|\tilde{y}_0-y^*\|<\gamma \text{ and } \|h^m(\tilde{y}_0)-y^*\|\geq \epsilon.
    \end{equation*}
    Here $h^m$ denotes the $m$-composition: $h=h\circ h \cdots \circ h$. We reduce $\epsilon>0$ in order that $\epsilon<r$. As a result, we can build a sequence of initial points $(\tilde{y}_n)_{n\in \mathbb{N}}$ that converges to $y^*$ and an integer sequence $(k_n)_{n\in\mathbb{N}}$ such that:
    \begin{equation*}
        \left\{
        \begin{array}{ll}
             \|h^m(\tilde{y}_n)-y^*\|<\epsilon \text{ for } 0\leq m< k_n,  \\
             \|h^{k_n}(\tilde{y}_n)-y^*\|\geq \epsilon. 
        \end{array}
        \right.
    \end{equation*}
    In other words, $k_n$ is the first time the trajectory starting from $\tilde{y}_n$ escapes the ball $B_{\epsilon}(y^*)$. \\
    
    We claim that there exists $\alpha>0$ such that $h\left(B_{\alpha}(y^*)\right) \subset B_{\epsilon/2}(y^*)$. By contradiction we have: 
    \begin{equation*}
        \forall \alpha>0, \exists x\in B_{\alpha}(y^*), h(x) \notin B_{\epsilon/2}(y^*).
    \end{equation*}
    Then we build a sequence $(x_n)_{n\in \mathbb{N}}$ that converges to $y^*$ such that: 
    \begin{equation*}
        \forall n\geq 0, \|h(x_n)-y^*\|\geq \frac{\epsilon}{2}. 
    \end{equation*}
    Given that the map $\tilde{\eta}$ is bounded by $\eta_0$:
    \begin{equation*}
        \|h(x_n)-y^*\| = \|x_n+\tilde{\eta}(x_n)F(x_n)-y^*\| \leq \|x_n-y^*\| + \eta_0\|F(x_n)\|.
    \end{equation*}
    As $F$ is continuous, $F(x_n) \to F(y^*)=0$ so $\|h(x_n)-y^*\| \to 0$. This is a contradiction with $\|h(x_n)-y^*\| \geq \frac{\epsilon}{2}$. \\

    As $\tilde{y}_n \to y^*$, we have:
    \begin{equation*}
        \exists n_0\in \mathbb{N}, \forall n\geq n_0, \|\tilde{y}_n-y^*\|<\alpha. 
    \end{equation*}
    Then for all $n\geq n_0$ we have: 
    \begin{equation*}
        \|h(\tilde{y}_n)-y^*\|<\frac{\epsilon}{2}. 
    \end{equation*}
    In particular, this means that $k_n>1$ for $n\geq n_0$. \\
    Now let us define the following sequence:
    \begin{equation*}
        u_n=h^{k_n-1}(\tilde{y}_n). 
    \end{equation*}
    By definition of $k_n$, $u_n \in B_{\epsilon}(y^*)$ for $n\geq n_0$ because at the time $k_{n-1}$, the trajectory starting from $\tilde{y}_n$ has not escaped the ball $B_{\epsilon}(y^*)$ yet. Consequently, the sequence $(u_n)_{n\in \mathbb{N}}$ is bounded and we can extract a convergent subsequence $(u_{\phi(n)})_{n\in \mathbb{N}}$. \\ 
    Denote by $u \in \bar{B}_{\epsilon}(y^*) \subset B_r(y^*)$ its limit. As $V$ is positive definite on $B_r(y^*)$ and $V$ is continuous we get:
    \begin{equation*}
        0=V(y^*) < V(u) = V\left(\displaystyle{\lim_{n\to +\infty}}u_{\phi(n)}\right) = \displaystyle{\lim_{n\to +\infty}} V(u_{\phi(n)}) = \displaystyle{\lim_{n\to +\infty}} V\left(h^{k_{\phi(n)}-1}(\tilde{y}_{\phi(n)})\right)
    \end{equation*}
    As $\dot{V}$ is negative the LCR algorithm makes the function $V$ decreases in the ball $B_r(y^*)$ so: $\displaystyle{\lim_{n\to +\infty}} V\left(h^{k_{\phi(n)}-1}(\tilde{y}_{\phi(n)})\right) \leq \displaystyle{\lim_{n\to +\infty}} V(\tilde{y}_{\phi(n)}) = V(y^*)$. This leads to the contradiction $V(y^*)<V(y^*)$. 
\end{proof}

Let us present some applications to classical algorithms(GD, RMSProp, pGD) with LCR backtracking.

\begin{example}
    \label{local_gd}[GD]
    Let $\R$ be differentiable. Consider the algorithm LCR with $F=-\nR$ and $V=\R$. Let $\theta^*$ be an isolated minimum of $\R$. Define the following translation of $V$: $\tilde{V}=\R(\theta)-\R(\theta^*)$ which gives $\dot{\tilde{V}}(\theta) = \dot{V}(\theta) = -\|\nR(\theta)\|^2$. So $V$ is definite positive and $\dot{V}$ is negative at the vicinity of $\theta^*$. Notice that in \cite{Absil} the stability of GD with descent conditions is only proved for analytic functions. \textbf{Here this assumption is not mandatory}. Besides we get the local attractivity of the optimizer thanks to theorem \ref{stability_theorem}.
\end{example}

\begin{example}
    \label{local_rms_prop}[RMSProp]
    Let us focus on RMSProp in the same configuration as in example \ref{ex_RMSProp_point}. According to theorem \ref{stability_theorem}, the Lyapunov control (defined in example \ref{ex_RMSProp_point}) makes it possible to stabilize the RMSProp with LCR backtracking.
    For this, consider the translation of $V$:
    \begin{equation*}
        \tilde{V}(s,\theta) = 2\left( \R(\theta)-\R(\theta^*) + \displaystyle{\sum_{i=1}^N \sqrt{\epsilon_a+s_i}}-N\sqrt{\epsilon_a}\right),
    \end{equation*}
    for a isolated minimum $\theta^*$ of $\R$. As in {example} \ref{ex_RMSProp_point}, $\tilde{V}$ is definite positive and $\dot{\tilde{V}}$ is negative at the neighborhood of $\theta^*$.
\end{example}

\begin{example}[pGD]
    Pursuing example \ref{ex_pGD_point}, the isolated local minima of $\R$ are stable by using the pGD with LCR backtracking. It is a direct consequence of theorem \ref{stability_theorem} applied to the translation $\tilde{V}(\theta) = V(\theta)-V(\theta^*)$ for $\theta^*$ a isolated minimum of $\R$. 
\end{example}

At this point, we would like to have a similar result for LMC. However,
in this case, the actual time step before the backtracking, does not only depend on a fix hyperparameter $\eta_{init}$ but on the previous learning rate. Then the application $h$ has to depend on the actual variable $y$ and the previous time step. Let us first define the map $\tilde{\eta}$ for LCM:
\begin{equation*}
        \left\{
        \begin{array}{ll}
        \tilde{\eta}(y,\eta) = \dfrac{f_2\eta}{f_1^p}, \\\\
         p = \min \left\{k\in \mathbb{N}, f\left(y,\dfrac{f_2\eta}{f_1^k}\right)\leq 0\right\}.
        \end{array}
        \right.
\end{equation*}
As we want to compose the map $h$, $m$ times, to get the $m$-th iterate of the algorithm it is necessary that $h$ has the same input and output dimension. Then we have to define $h$ as follows:
\begin{equation*}
    h(y,\eta) = 
    \begin{pmatrix}
        y+\tilde{\eta}(y,\eta) F(y) \\
        \tilde{\eta}(y,\eta)
    \end{pmatrix}
    .
\end{equation*}
As the map has now two arguments (defined on $\Rb^m \times \Rb_+$), one can wonder what the notion of classic Lyapunov stability means in that case. Stability is measured respect to a fixed point $(y^*, \eta^*)$ of $h$ but we are not interested in remaining close to a learning rate $\eta^*$. This is why the notion of stability is too constraining and we have to rely on partial stability \cite{vorotnikov_partial} because we are only interested on the trajectory $(y_n)_{n\in \mathbb{N}}$. In fact, defining the sequences $(y_n)_{n\in \mathbb{N}}$ and $(\eta_n)_{n\in \mathbb{N}}$ by $(y_{n+1},\eta_{n+1})=h(y_n,\eta_n)$, we say that $y^*$ is $y$-stable on the whole with respect to $\eta$ if:
\begin{equation*}
    \forall \epsilon>0, \exists \delta>0, \|y_0-y^*\|<\delta \implies \forall n\geq 0, \|y_n-y^*\|< \epsilon.
\end{equation*}
This definition appears as a sort of projection of the classic stability on the $y$-axis and we can think that we can reproduce the previous proof for LCR with this new definition. However, a crucial fact in the previous proof was the boundedness of the map $\tilde{\eta}$ but in the case of LCM, we have not found any reasons for this map to be bounded. This is why, we only deal with the stability of LCR with respect to a isolated local minimum in this paper. The stability of LCM remains an open question. Nevertheless, we will see in the next subsection, that both LCR and LCM are stable with respect to global minima, which is of greater interest in practice.  

\subsection{Stability of the set of global minima}

The stability of local minima proves that LCR preserves an important dynamic property of the ODE. But in practice, the stability of interest concerns the set of \textbf{global minima} $\globalset$:
\begin{equation*}
    \globalset \coloneqq \{\theta^* \in \Rb^N; \forall \theta \in \Rb^N, \R(\theta) \geq \R(\theta^*)\}. 
\end{equation*}
Indeed, we want to avoid the situation where the initial point is near a global minimum but converges to a local one or a saddle point. Indeed, such an undesirable behavior has been observed in practice for other optimizers (see \citet{Bilel} for a illustration). 
In this subsection, we will establish conditions under which the set $\globalset$ is stable and attractive. We will focus on a class of maps called KL functions. We refer to appendix B for their definitions. KL functions are involved in many optimization problems \cite{Bolte_KL, Bolte_semi_analytic} because they include semi-algebraic, semi-analytic functions and especially analytic ones. In neural networks optimization, many error functions satisfy this hypothesis because the activation functions are often analytic, such as sigmoid, tanh, Gelu \cite{gelu} or Silu for instance. 
Note that the latter two are regularizations of relu which is not differentiable.
We need this kind of functions to avoid a situation where there exists a local minimum or a saddle point arbitrary close to $\globalset$. This hypothesis will "force $\globalset$ to be isolated in a way" that we will clarify in the proof.\\

\begin{theorem}
    Denote by $\globalset_V$ the set of global minima of $V$. $V$ is supposed to be differentiable, positive and its derivative $\dot{V}$ negative on $\Rb^m$. If $V$ is radially unbounded, then $\globalset_V$ is stable for LCR and LCM:
    \begin{equation*}
        \forall \epsilon>0, \exists r>0, d(y_0,\globalset_V)<r \implies \forall n\geq 0, d(y_n,\globalset_V)<\epsilon.
    \end{equation*}
    Moreover, if we also assume that $V$ is a KL function and ($\dot{V}=0\implies \nabla V=0$), then $\globalset_V$ is attractive for LCR and LCM:
    \begin{equation*}
             \exists \gamma>0, \forall y_0 \in \Rb^m, d(y_0,\globalset_V)<\gamma \implies \displaystyle{\lim_{n\to +\infty}}d(y_n,\globalset_V)=0.      
    \end{equation*}
\end{theorem}

\begin{proof}
    Without loss of generality we can assume that the global minimum value of $V$ is zero. We can note that $\globalset_V$ is compact. \\ 
    To show that it is closed, let us take a convergent sequence $(y_n)_{n\in \mathbb{N}} \in \globalset_V^{\mathbb{N}}$: $y_n \to y^*$. For all $n\geq 0$, $y_n \in \globalset_V$ so $V(y_n)=0$ and by continuity of $V$: $V(y^*)=0$. Then $y^* \in \globalset_V$. If $\globalset_V$ is not bounded, there exists a sequence $(y_n)_{n\in \mathbb{N}} \in \globalset_V^{\mathbb{N}}$ such that $y_n \to +\infty$. As $V$ is radially unbounded, $V(y_n) \to \infty$. This is a contradiction since for all $n\geq 0$, $V(y_n)=0$. \\
    Now we will deal with the stability. To do so, let us prove the existence of two maps $\alpha_1$ and $\alpha_2$ which are continuous and strictly increasing satisfying: 
    \begin{equation*}
        \left\{
        \begin{array}{ll}
             \alpha_1(0)=\alpha_2(0)=0,  \\\\
             \forall y\in \Rb^m, \alpha_1(d(y,\globalset_V)) \leq V(y) \leq \alpha_2(d(y,\globalset_V)).
        \end{array}
        \right.
    \end{equation*}
    Define the increasing map:
    \begin{equation*}
        \psi(s)=\inf \{V(y), d(y,\globalset_V) \geq s\},
    \end{equation*}
    that exists because $V$ is lower bounded by zero. As $V$ is continuous, $\psi$ is continuous. Since $\globalset_V$ is closed, $d(y,\globalset_V)=0 \Leftrightarrow y\in \globalset_V$. Then, for $s>0$ $\psi(s)>0$. Then we can build a map $\alpha_1$ strictly increasing with $\alpha_1(0)=0$ such that:
    \begin{equation*}
        \forall s\geq 0, \frac{\psi(s)}{2} \geq \alpha_1(s). 
    \end{equation*}
    By definition of $\psi$, we get: 
    \begin{equation*}
        V(y) \geq \psi(d(y,\globalset_V)) \geq \alpha_1(d(y,\globalset_V)).
    \end{equation*}
    For the second map, let us note that the set $\{y\in \Rb^m, d(y,\globalset_V)\leq s\}$ is compact since $\globalset_V$ is compact, so the map $\tilde{\psi}(s)=\sup \{V(y), d(y,\globalset_V)\leq s\}$ is well-defined. We build $\alpha_2$ as below by imposing that $2\tilde{\psi}(s) \leq \alpha_2(s)$ for all $s\geq 0$. \\\\
    To prove the stability, let $\epsilon>0$. Define $r=\alpha_2^{-1}\circ \alpha_1(\epsilon)$. As $\dot{V}$ is negative, the sequence $(V(y_n))_{n\in \mathbb{N}}$ is decreasing and we can write:
    \begin{equation*}
        \alpha_1(d(y_n,\globalset_V)) \leq V(y_n) \leq V(y_0) \leq \alpha_2(d(y_0,\globalset_V)).
    \end{equation*}
    As a result, we get:
    \begin{equation*}
        d(y_n,\globalset_V) \leq (\alpha_1^{-1}\circ \alpha_2)(d(y_0,\globalset_V)) \leq (\alpha_1\circ \alpha_2)(\epsilon) = \epsilon.
    \end{equation*}
    Now let us tackle the attractivity.
    Denote by $\globalset_{\epsilon}$ the $\epsilon$-neighborhood of $\globalset_V$: 
    \begin{equation*}
        \globalset_{\epsilon}=\{y \in \Rb^m, d(\globalset_V,y)<\epsilon\}. 
    \end{equation*}
    Let us first prove the following: 
    \begin{equation*}
        \exists \gamma>0, \forall y\in \globalset_{\gamma}\setminus \globalset_V, \dot{V}(y)\neq 0. 
    \end{equation*}
    By contradiction, we can build a sequence $(y_n)_{n\in \mathbb{N}}$ such that:
    \begin{equation*}
        \left\{
        \begin{array}{ll}
             \forall n\geq 0, y_n \notin \globalset_V,  \\
              \forall n\geq 0, \dot{V}(y_n)=0, \\
              d(y_n,\globalset_V) \to 0.
        \end{array}
        \right.
    \end{equation*}
    As $\globalset_V$ is bounded, the sequence $(y_n)_{n\in \mathbb{N}}$ is bounded and we can extract a convergent subsequence: $y_{\psi(n)} \to y^* \in \globalset_V$. So there exists $n_0\geq 0$ such that $y_{\psi(n_0)}$ lies in the neighborhood $U$ of $y^*$ where we can apply the KL inequality ($y^*$ is a critical point):
    \begin{equation*}
        \|\nabla V(y_{\psi(n_0)})\| \geq \dfrac{1}{\phi'\left(V(y_{\psi(n_0)}). \right)}.
    \end{equation*}
    As $\phi'>0$, $\nabla V(y_{\psi(n_0)}) \neq 0$ so $\dot{V}(y_{\psi(n_0)}) \neq 0$, as ($\dot{V}=0\implies \nabla V=0$). This is a contradiction since $\dot{V}(y_{\psi(n_0)})=0$.\\
    Now we consider $(y_n)_{n\in \mathbb{N}}$ the sequence generated by LCR or LCM. First, let us show that $d(y_n,\zerodot) \to 0$. \\ 
    By contradiction assume that $\limsup d(y_n,\zerodot)>0$. Then:
    \begin{equation*}
        \exists \epsilon>0, \exists y_{\psi(n)}, \forall n\geq 0, d(y_{\psi(n)},\zerodot) \geq \epsilon. 
    \end{equation*}
    By Theorems \ref{LCR_accumulation} and \ref{LCM_accumulation}, there exists a convergent subsequence $(y_{\psi\circ\phi(n)})_{n\in \mathbb{N}}$ such that $y_{\psi\circ\phi(n)} \to y^* \in \zerodot$. This contradicts the inequality $d(y_{\psi(n)},\mathcal{\zerodot}) \geq \epsilon$.\\
    To conclude apply the stability definition with $\epsilon=\frac{\gamma}{2}$ to obtain a real $r>0$ such that: $d(y_0,\globalset_V)<r \implies \forall n\geq 0, d(y_n,\globalset_V)<\frac{\gamma}{2}$. The trajectory stays in $\globalset_{\gamma/2}$ and we have proved that $\left(\globalset_{\gamma/2}\setminus \globalset_V \right)\cap \zerodot=\emptyset$ which is sufficient to get the attractivity. 
\end{proof}

\begin{example}
    In the same manner than the previous subsection, this theorem can be applied with the optimizers of example \ref{local_gd} and \ref{local_rms_prop}. For RMSProp, we have: $\forall \theta\in \Rb^N, \forall s\in \Rb_+^N$, $V(s,\theta) \geq \R(\theta)$ so if we reach the global minima of $V$, we have attained the global minima of $\R$. In example \ref{ex_Momentum_point}, $V$ is positive, radially unbounded and $\dot{V}$ is negative so the set $\globalset_V$ is stable. However, the attractive part can not be applied since $\dot{V}=0$ does not imply $\nabla V=0$. 
\end{example}

\section{Global convergence of LCR and LCM}
\label{section_cv}

In the last sections, attractivity and stability properties of the algorithms LCR and LCM have been discussed. These qualitative behaviors are essential to get a good optimizer. In particular, the examples of the previous section allow stating the stability of the set of global minima for several optimizers combined with LCR and LCM backtracking: if they are initialized in a neighborhood of this set, the sequence generated by the algorithm converges to $\globalset_V$. But what happens if the initialization is far away from this neighborhood? This is the problem of the "global" convergence of the process towards a critical point. \\

First in subsection \ref{case_gd}, a convergence result stated in \cite{Absil} and \cite{Rondepierre} for descent algorithms (GD) is investigated in the sense that we obtain convergence rates. Finally, in subsection \ref{abstract_rez}, an abstract convergence framework concerning 
LCR or LCM is introduced with an interesting application. From now on, LCR and LCM will be treated as one.

\subsection{The case of GD}
\label{case_gd}

The authors in \cite{Rondepierre} extends the result of \cite{Absil} from Lojasiewicz to KL functions in the case where $F=\R$ and $V=\R$. We will add to this theorem an estimation of the convergence rate. 
In practice, this is done by following the first same steps as in \cite{Absil}. To compute this estimation, we first need to generalize a discrete Gronwall lemma stated below:    

\begin{lemma}[Non-Linear Gronwall Lemma]
    Let $(v_n){_n\in\mathbb{N}}$ be a positive sequence and $(u_n){_n\in\mathbb{N}}$ such that for all $n \in \mathbb{N}, u_{n+1}-u_n \leq -v_n \psi(u_n)$ where $\psi$ is a continuous strictly positive and increasing function. Then:
    \begin{equation*}
        u_n \leq \Psi^{-1}\left( \Psi(u_0) - \displaystyle{\sum_{k=0}^{n-1}}v_k\right),
    \end{equation*}
    where $\Psi$ is a primitive of $\frac{1}{\psi}$.
    \label{bihari_lemma}
\end{lemma}

\begin{proof}
    As $\psi$ is increasing and by invoking the mean value theorem, there exists $c\in \Rb$ such that $u_{n+1} \leq c \leq u_n$ satisfying:
    \begin{equation*}
        \Psi(u_{n+1}) - \Psi(u_n) = \int_{u_n}^{u_{n+1}} \frac{1}{\psi(x)}dx = \dfrac{u_{n+1}-u_n}{\psi(c)} \leq \dfrac{-v_n \psi(u_n)}{\psi(u_{n+1})} \leq -v_n.
    \end{equation*}
    We sum the term for $k=0$ to $k=n-1$ to obtain:
    \begin{equation*}
        \Psi(u_n) \leq \Psi(u_0)-\displaystyle{\sum_{k=0}^{n-1}}v_k.
    \end{equation*}
    Furthermore, $\Psi$ is a continuous strictly increasing function hence it admits an increasing inverse $\Psi^{-1}$ on its codomain. Applying $\Psi^{-1}$ on the previous inequality yields the result.
\end{proof}

Equipped with this inequality, we ca now give a 
detailed result about the convergence of GD with backtracking.  

\begin{theorem}[GD Estimation]
    \label{estimate_theta}
    Let $\R\in \mathcal{C}^1(\Rb^N)$ and $\nR$ locally Lipshitz. Assume that $\R$ satisfies the KL condition at the neighborhood of {its} critical points.
    Then, either $\displaystyle{\lim_{n\to + \infty}} \|\theta_n\|=+\infty$, or there exists $\theta^* \in \mathcal{E}$ such that:
    \begin{equation*}
        \displaystyle{\lim_{n\to + \infty}} \|\theta_n-\theta^*\|=0.
    \end{equation*}
    Furthermore there exists $n_0$ such that for all $n \geq n_0$:
    \begin{equation}
        \|\theta_n-\theta^*\| \leq \frac{1}{\lambda} \phi\left(\Psi^{-1}\left(\Psi(\R(\theta_0)-\R(\theta^*)) - \displaystyle{\sum_{k=0}^{n-1}} \eta_k\right)\right),
        \label{convergence_theta LC_EGD2}
    \end{equation}
     where $\Psi$ is a primitive of $\phi'^2$ on $]0,\gamma[$ and $\gamma$ is given 
     in the definition \eqref{KL_inequality} of KL functions.
\end{theorem}

\begin{proof}
    Let us detail the case when $(\theta_n)_{n\in\mathbb{N}}$ does not diverge to infinity. Then the sequence is bounded and admits an accumulation point $\theta^* \in \mathcal{E}$ according to proposition \ref{LCEGD_accumulation}. The goal is to prove that $\theta_n \to \theta^*$.\\
    The sequence $(\R(\theta_n))_{n\in\mathbb{N}}$ is decreasing and {lower} bounded by 0. So it converges to some real $l$. Without loss of generality let us assume that $l=\R(\theta^*)=0$. If the sequence $(\theta_n)_{n\in\mathbb{N}}$ is eventually constant then the result is straightforward. Otherwise let us remove all the indices such that $\theta_{n+1}=\theta_n$. \\
    Now note that:
    \begin{equation*}
        \R(\theta_{n+1})=\R(\theta_n) \Rightarrow \theta_{n+1}=\theta_n.
    \end{equation*}
    Indeed, we have $\R(\theta_{n+1})-\R(\theta_n)\leq -\lambda \eta_n \|\nR(\theta_n)\|^2 \leq 0$. As $\R(\theta_{n+1})=\R(\theta_n)$ it follows $\eta_n\nR(\theta_n)=0$. As $\eta_n$ is strictly positive it comes: $\nR(\theta_n)=0$. Then $\theta_{n+1}=\theta_n$.\\
    As a result $\R(\theta_n)$ is strictly decreasing and $\R(\theta_n)>0$. \\
    
    Provided that $\theta_n \in \mathcal{U}$ (where $\mathcal{U}$ is the neighborhood of $\theta^*$ where we can apply the KL inequality), we have:
    \begin{equation*}
        \R(\theta_n)-\R(\theta_{n+1}) \geq \lambda \|\nR(\theta_n)\|\|\theta_{n+1}-\theta_n\| \geq \lambda \|\theta_{n+1}-\theta_n\| \dfrac{1}{\phi'(\R(\theta_n))}.
    \end{equation*}
    Given that $\phi'(\R(\theta_n))>0$:
    \begin{equation}
        \label{eq:in1}
        \|\theta_{n+1}-\theta_n\| \leq \phi'(\R(\theta_n)) \dfrac{\R(\theta_n)-\R(\theta_{n+1})}{\lambda}.
    \end{equation}
    Since $\forall x \in \left[ \R(\theta_{n+1}),\R(\theta_n) \right]$, $\phi'(\R(\theta_n)) \leq \phi'(x)$ (as $\phi$ is concave), we deduce:
    \begin{multline*}
    \left(\R(\theta_n)-\R(\theta_{n+1})\right) \phi'(\R(\theta_n)) = 
    \displaystyle{\int_{\R(\theta_{n+1})}^{\R(\theta_n)}} \phi'(\R(\theta_n)) dx \leq \displaystyle{\int_{\R(\theta_{n+1})}^{\R(\theta_n)}} \phi'(x)dx = \phi(\R(\theta_n))-\phi(\R(\theta_{n+1})).
    \end{multline*}
    Then, provided $\theta_n \in \mathcal{U}$, using this last inequality in conjunction with \eqref{eq:in1}, we get:
    \begin{equation*}
        \|\theta_{n+1}-\theta_n\| \leq \dfrac{\phi(\R(\theta_n))-\phi(\R(\theta_{n+1}))}{\lambda}.
    \end{equation*}
Given that $p>q$ such that $\theta_p, \cdots, \theta_{q-1} \in \mathcal{U}$ it follows:
\begin{equation*}
    \displaystyle{\sum_{n=p}^{q-1}}\|\theta_{n+1}-\theta_n\| \leq \dfrac{\phi(\R(\theta_p))-\phi(\R(\theta_{q}))}{\lambda}.
\end{equation*}
Now, let $r>0$ be such that $B_r(\theta^*) \subset \mathcal{U}$. Given that $\theta^*$ is an accumulation point of $(\theta_n)_{n\in\mathbb{N}}$ and $(\phi(\R(\theta_n)))_{n\in\mathbb{N}}$ converges to $\phi(0)=0$, there exists $n_0$ such that:
\begin{equation*}
    \begin{array}{lll}
    \|\theta_{n_0}-\theta^*\|<\frac{r}{2}, \\\\
    \forall q \geq n_0: \frac{1}{\lambda}\left[\phi(\R(\theta_{n_0}))-\phi(\R(\theta_{q}))\right] < \frac{r}{2}.
    \end{array}
\end{equation*}

It remains to show that $\theta_n \in B_r(\theta^*)$ for all $n>n_0$. By contradiction assume that it is not the case: $\exists n>n_0, \theta_n \notin B_r(\theta^*)$. The set $\left\{ n>n_0, \theta_n \notin B_r(\theta^*)\right\}$ is a non empty {set} bounded by below, subset of $\mathbb{N}$. So we can consider the minimum of this set that we denote by $p$. As a result, $\forall n_0 \leq n < p$, $\theta_n \in \mathcal{U}$ so:
\begin{equation*}
    \displaystyle{\sum_{n=n_0}^{p-1}}\|\theta_{n+1}-\theta_n\| \leq \frac{1}{\lambda} \left[\phi(\R(\theta_{n_0}))-\phi(\R(\theta_{p}))\right] < \frac{r}{2}. 
\end{equation*}
This implies:
\begin{equation*}
    \| \theta_p-\theta^* \| \leq \displaystyle{\sum_{n=n_0}^{p-1}}\| \theta_{n+1}-\theta_n \| + \| \theta_{n_0}-\theta^* \| < r.
    \label{decomposition_theta}
\end{equation*}
This is a contradiction because $\| \theta_p-\theta^* \| \geq r$. As $r$ is arbitrarily small this shows the convergence.\\

Now, in order to obtain the bound on the convergence speed, let us rewrite $\|\theta_n-\theta^*\|$ for $n\geq n_0$:
\begin{equation}
    \|\theta_n-\theta^*\| = \left|\left|\displaystyle{\sum_{k=n}^{+\infty}} (\theta_k-\theta_{k+1})\right|\right| \leq \displaystyle{\sum_{k=n}^{+\infty}} \|\theta_{k+1}-\theta_k\| \leq \frac{1}{\lambda} \displaystyle{\lim_{q \to +\infty}} \left[ \phi(\R(\theta_n))-\phi(\R(\theta_q))\right] = \dfrac{\phi(\R(\theta_n))}{\lambda}.
    \label{theta_R}
\end{equation}

The KL inequality and the dissipation inequality \eqref{dissipation_inequality} are verified for $n\geq n_0$ so:
    \begin{equation*}
        \R(\theta_{n+1}) - \R(\theta_n) \leq -\lambda \eta_n \|\nR(\theta_n)\|^2 \leq -\lambda \eta_n \dfrac{1}{\phi'\left(\R(\theta_n)\right)^2}.
    \end{equation*}
    By applying lemma \ref{bihari_lemma} to the previous inequality with $v_n=\lambda \eta_n$, $u_n=\R(\theta_n)$ and $\psi(x) =\phi'(x)^{-2}$, since $\phi$ is concave strictly increasing, we get:
    \begin{equation}
        \R(\theta_n) \leq \left(\Psi^{-1}\left(\Psi(\R(\theta_0)-\R(\theta^*)) - \displaystyle{\sum_{k=0}^{n-1}} \eta_k\right)\right).
        \label{R_convergence}
    \end{equation}
    The convergence rates come directly by combining inequality \eqref{theta_R} with \eqref{R_convergence} because $\phi$ is increasing. 
\end{proof}

\subsection{A general abstract result}
\label{abstract_rez}

From the previous proof for GD, we can extract an abstract structure that allows to deduce convergence results not restricted to GD. This is the object of the next theorem. It is stated in the Lojasiewicz setting for simplification and the proof, very similar to the one of GD, is presented in appendix \ref{proof_cv}. It can be considered the discrete counterpart of the convergence theorem in \cite{Haraux_autonomous}.

\begin{theorem}[Convergence Framework]
    Let $V \in \mathcal{C}^2(\Rb^m,\Rb^+)$ such that $\dot{V}\leq 0$ and $\left(\nabla V(y)=0 \Rightarrow F(y)=0\right)$. Assume that for all points $y^* \in \mathcal{Z}$, there exists a neighborhood $\mathcal{U}$ of $y^*$, $\gamma \geq 0$, $\alpha_1 \in ]0,1[$, $c,c_1>0$ such that:
    \begin{enumerate}[label=(\roman*)]
        \item $\forall y \in \mathcal{U}$, $\dot{V}(y) \leq -c \|\nabla V(y)\|^{\gamma} \|F(y)\|$,
        \item $\forall y \in \mathcal{U}$, $\|\nabla V(y)\| \geq c_1 (V(y)-V(y^*))^{1-\alpha_1}$,
        \item $\gamma (1-\alpha_1)<1$.
    \end{enumerate}
    Then if we consider algorithms LCR and LCM with the backtracking $V$, all bounded sequences $(y_n){_{n\in\mathbb{N}}}$ converges to $y^* \in \mathcal{Z}$. Assume in addition that there exists $c_2>0$ and $\alpha_2 \leq 1$ satisfying:
    \begin{equation}
        \forall y \in \mathcal{U} \text{ , }\|F(y)\| \geq c_2 (V(y)-V(y^*))^{1-\alpha_2}
        \label{condition_vitesse}
    \end{equation}
    Then there exists $n_0 \in \mathbb{N}$ such that for all $n \geq n_0$ we have:
    \begin{itemize}
        \item If $\dfrac{\alpha_2}{1-\alpha_1} < \gamma$ (\textbf{subexponential}):
        \begin{equation*}
        \|y_n-y^*\| \leq \dfrac{C_1}{\left[\left(V(y_{n_0})-V(y^*)\right)^{\alpha_2-\gamma(1-\alpha_1)}+C_2\displaystyle{\sum_{k=n_0}^{n-1}}\eta_k \right]^{\frac{1-\gamma(1-\alpha_1)}{\gamma(1-\alpha_1)-\alpha_2}}}.
        \label{estimation_puissance}
    \end{equation*}
    \item If $\dfrac{\alpha_2}{1-\alpha_1} = \gamma$ (\textbf{exponential}):
     \begin{equation*}
        \|y_n-y^*\| \leq C_1 \left(V(y_{n_0})-V(y^*)\right)^{1-\gamma(1-\alpha_1)} \exp{ \left( -C_3\displaystyle{\sum_{k=n_0}^{n-1}}\eta_k\right)}.
        \label{estimation_exp}
    \end{equation*}
    \item Finally if $\dfrac{\alpha_2}{1-\alpha_1} > \gamma$ and the sum of $(\eta_k)_{k \in \mathbb{N}}$ diverges, the sequence $(y_n)_{n\in\mathbb{N}}$ converges in \textbf{finite time}, that is to say, $y_n=y^*$ if $n$ satisfies the following inequality:
    \begin{equation*}
        \displaystyle{\sum_{k=n_0}^{n-1}}\eta_k \geq \dfrac{\left(V(y_{n_0})-V(y^*)\right)^{\alpha_2-\gamma(1-\alpha_1)}}{-C_2},
    \end{equation*}
    \end{itemize}
    where:
    \begin{equation*}
        \left\{
        \begin{array}{cc}
            C_1 = \dfrac{1}{\lambda c c_1^{\gamma} \left(1-\gamma(1-\alpha_1)\right)}, \\\\
            C_2 = \lambda cc_1^{\gamma}c_2\left[\gamma(1-\alpha_1)-\alpha_2\right], \\\\
            C_3 = \frac{c_2}{C_1}.
        \end{array}
        \right.
    \end{equation*}
    \label{abstract_cv}
\end{theorem}

\begin{remarknot}
    As in the case of KL inequality \eqref{convergence_theta LC_EGD2} the different inequalities can hold only for $y \in \mathcal{U}\setminus \{y^*\}$.
\end{remarknot}
As an application of this theorem, an in-depth study of the parameter $p>1$ for pGD is proposed.  
\begin{example}
    \label{pgd_cv}
    Assume $\R \in \mathcal{C}^2(\Rb^N)$ and Lojasiewicz at any critical point. Let us take the same Lyapunov function as in the example \ref{ex_pGD_point} for the pGD update. We have already shown that $\dot{V}(\theta)\leq 0$ for all $\theta \in \Rb^N$. The implication $\left(\nabla V=0 \implies F=0\right)$ holds. The condition i) becomes an equality:
    \begin{equation*}
        \dot{V}(\theta) = - \|\nR(\theta)\|^{\frac{p}{p-1}} = - \|\nR(\theta)\|\|\nR(\theta)\|^{\frac{1}{p-1}} = -\|\nabla V(\theta)\| \|F(\theta)\|.
    \end{equation*}
    So $c=1$ and $\gamma=1$. The second condition ii) is just the Lojasiewicz inequality where $\alpha_1$ denotes the Lojasiewicz coefficient, $c$ the Lojasiewicz constant and $\mathcal{U}$ the neighborhood of a critical point of $\R$ where the Lojasiewicz inequality is valid. As $\alpha_1<1$ the third condition iii) is clearly true. Hence the \textbf{convergence of the sequence $(\theta_n)_{n\geq 0}$ to a point of $\mathcal{Z}=\mathcal{E}$ is insured}. Now it remains to handle the inequality \eqref{condition_vitesse} for $\theta \in \mathcal{U}$:
    \begin{equation*}
        \|F(\theta)\| = \|\nR(\theta)\|^{\frac{1}{p-1}} \geq c_1^{\frac{1}{p-1}} \left[\R(\theta)-R(\theta^*)\right]^{\frac{1-\alpha_1}{p-1}} = c_1^{\frac{1}{p-1}} \left[V(\theta)-V(\theta^*)\right]^{\frac{1-\alpha_1}{p-1}}.
    \end{equation*}
    by using again Lojasiewicz inequality. Then $c_2 = c_1^{\frac{1}{p-1}}>0$ and $\alpha_2 = \dfrac{\alpha_1+p-2}{p-1} \leq 1$ because $\alpha_1<1$. Therefore by comparing $\frac{\alpha_2}{1-\alpha_1}$ to 1 we obtain:
    \begin{itemize}
        \item If $\alpha_1<\frac{1}{p}$, the convergence is subexponential.
        \item If $\alpha_1 = \frac{1}{p}$, the convergence is exponential.
        \item If $\alpha_1 > \frac{1}{p}$ and the sum of time steps diverges, the sequence $(\theta_n)_{n\geq 0}$ converges in finite time. 
    \end{itemize}
\end{example}

\section{Conclusion}\label{conclusion}

In this paper, we investigate two non constant step size policies applied to any ML optimizer that can be considered as the discretisation of an ODE (GD, pGD, Momentum, RMSProp,...). These policies can be seen as the generalization of the backtracking Armijo rule in the optimization community or as Lyapunov stabilization in the control theory. \\
In this framework, the most challenging part concerns the localization of the accumulation points of the sequence generated by the optimizer. This fact seems obvious when the time step is constant and is well documented \cite{LaSalle_discrete, Lyapunov_discrete, Lyapunov_inria, Lyapunov_non_monotonic}, when the sequence is generated by a continuous map (the function $h$ in \eqref{h} is continuous). But our time step policies are far from being continuous and we have to use recent results of selection theory to overcome this problem. \\
Despite this supplementary technical difficulty, these strategies have great qualitative properties: local stability of isolated and global minimums 
and strong global convergence (convergence of the sequence of iterates) to the set where the Lyapunov derivative is zero (for some ODEs, this set is exactly the set of stationary points). This holds for \textbf{any choice of hyperparameters}. This is precisely the main benefit of these methods since contrary to constant step size algorithms, the user does not have to tune hyperparameters (this may be very hard, see \cite{Lipshitz_constant}) to ensure these properties.\\ 
Finally, asymptotic convergences rates are derived depending on the Lojasiewicz coefficient and lead to an exhaustive study of the asymptotic behavior of the pGD optimizer. \\
Some questions remain still open for these backtracking policies:
\begin{itemize}
    \item Could we expect that all the accumulation points of LCM lie in $\zerodot$? This problem is closely related to the construction of a continuous selection on the whole domain $\Rb^m$ (see remark \ref{remark_LCM_zerodot}).
    \item Could we prove a partial stability result for LCM (see the discussion at the end of the subsection \ref{local_stability})?
    \item For some ODEs, the convergence to $\zerodot$ implies the convergence of the variable of interest $(\theta_n)_{n\in \mathbb{N}}$ to a first order stationary point: it can be a local minimum/maximum or a saddle point. An interesting research perspective will be to investigate the convergence to local minimums that has been done in \cite{sd_gd_escape} for constant step size gradient descent. 
\end{itemize}

\backmatter

\section*{Funding}
This work was supported by CEA/CESTA and LRC Anabase. 

\section*{Conflict of interest}
The authors declare that they have no conflict of interest.

\begin{appendices}

\section{Selection theory: notions and some theorems}
\label{selection_theory}

The next definitions and theorems about selection theory are mainly taken from \cite{infinite_dimensional_analysis,selection_book}. Selection theory concerns set value map between topological spaces. Fixing two topological spaces $X$ and $Y$ consider a set value map $\phi: X \mapsto \mathcal{P}(Y)$. A selection of $\phi$ is a map $s:X \mapsto Y$ such that:
\begin{equation*}
    \forall x \in X, s(x) \in \phi(x).
\end{equation*}
If $\phi$ does not have the empty set as a value, this application $s$ exists due to the axiom of choice. The main goal of selection theory is to give conditions on $\phi$ in order to have the existence of a selection having some interesting properties such as measurability or continuity.\\
To do this, the notion of continuity has to be generalized to set value mappings. One of this most useful generalization is the so called lower hemicontinuity (or semicontinuity) stated below:

\begin{definition}
    \label{hemicontinuity}
    Let $x\in X$. The map $\phi$ is lower hemicontinuous at $x$ if for every open set $U$ that meets $\phi(x)$ ($\phi(x) \cap U \neq \emptyset)$, there is a neighborhood $\voisi$ of x such that:
    \begin{equation*}
        z \in \voisi \implies \phi(z) \cap U \neq \emptyset.
    \end{equation*}
    The map $\phi$ is lower hemicontinuous on $X$ if it is lower hemicontinuous at each point of $X$.
\end{definition}

The first and most famous selection theorem is due to Michael in 1956 and lower hemicontinuity is at the center of this theorem (see theorem 17.66 p589 in \cite{infinite_dimensional_analysis}):
\begin{theorem}[Michael's selection theorem]
    A lower hemicontinuous set value map $\phi$ from a paracompact space into a Fréchet space with non empty closed convex values admits a continuous selection. 
    \label{Michael_selection}
\end{theorem}

In our problem $X=\Rb^N\setminus \mathcal{Z}$ (paracompact because it is metrizable), $\phi=T$ and we can consider $Y=\Rb_+^*$ or $Y=\Rb$. In the first case $Y$ is not a Fréchet space and in the second there is no evidence that $T(y)$ is closed for the euclidean topology on $\Rb$. That is why a more recent and general theorem has to be used (see theorem 6.2 p.116 in \cite{selection_book} with its proof).
\begin{theorem}
    Let $\mathcal{O}$ be a nonempty open subset of a Banach space $B$. Then every \textbf{lower hemicontinuous} map $\phi:X \mapsto \mathcal{O}$ from a paracompact space $X$ with \textbf{convex, closed (in $\mathcal{O}$)} values $\phi(x)$, $x\in X$, admits a continuous selection.  
    \label{theorem_used}
\end{theorem}
In our case $B=\Rb$ is a Banach space and $\mathcal{O}=\Rb_+^*$ is an open space of $\Rb$. This theorem is adapted to the problem of section \ref{limit_points}. 

\section{KL functions}

In this section, let us recall definitions and main results about Kurdyka-Lojasiewicz (KL) functions that are widely used in non-convex optimization. 
KL inequality, stated below, {gives insights on} the behavior of a function around its critical points.

\begin{definition}[KL]
    We say that a differentiable function $g: \mathbb{R}^m \mapsto \mathbb{R}$ is KL at a critical point $y^*\in\critic_g$ if there exists $\gamma>0$, $\voisi$ a neighborhood of $y^*$ and a continuous concave function $\phi: [0,\gamma] \mapsto [0,+\infty[$ such that:
    \begin{enumerate}
        \item $\phi(0)=0$, $\phi \in \mathcal{C}^1(]0,\gamma[)$ and $\phi'>0$ on $]0,\gamma[$.
        \item For all $y \in \mathcal{U}_{y^*} \coloneqq \voisi \cap \{ y\in \Rb^m, g(y^*)<g(y)<g(y^*)+\gamma \}$:
        \begin{equation*}
            \phi'\left(g(y)-g(y^*)\right) \|\nabla g(y)\| \geq 1.
        \end{equation*}
    \end{enumerate}
    \label{KL_inequality}
\end{definition}

\begin{remarknot}
    We will omit to mention the point $y^*$ for $\mathcal{U}_{y^*}$ when it is clear denoting it simply by $\mathcal{U}$. 
\end{remarknot}

A particular case of this inequality which is the most useful in practice is called simply Lojasiewicz inequality, see \cite{Loj1} and \cite{Loj2}:

\begin{definition}[Lojasiewicz]
     We say that a differentiable function $g: \Rb^m \mapsto \Rb$ satisfies Lojasiewicz inequality at a critical point $y^*\in\critic_g$ if there are $c>0$, $\sigma>0$ and $0 < \alpha \leq 1$ such that:
    \begin{equation*}
        \|y-y^*\|<\sigma \Rightarrow \|\nabla g(y)\| \geq c \|g(y)-g(y^*)\|^{1-\alpha}.
        \label{loj_condition}
    \end{equation*}
\end{definition}

\begin{remarknot}
    The Lojasiewicz inequality is a particular case where $\phi(x)=c\frac{x^{\alpha}}{\alpha}$.
\end{remarknot}

The fundamental theorem due to Lojasiewicz \cite{Lojasiewicz_gradient}, that justifies the crucial role of this class of functions, says that analytic functions satisfy Lojasiewicz inequality at the neighborhood of all their critical points:

\begin{theorem}[Lojasiewicz]
    Let $g: \Rb^m \mapsto \Rb$ be an analytic function. Then for all critical point $y^*\in\critic_g$ of $g$, there are $c>0$, $\sigma>0$ and $0 < \alpha \leq \frac{1}{2}$ such that:
    \begin{equation*}
        \|y-y^*\|<\sigma \Rightarrow \|\nabla g(y)\| \geq c \|g(y)-g(y^*)\|^{1-\alpha}.
    \end{equation*}
    \label{loj_theorem}
\end{theorem}

\section{Proof of the convergence theorem}
\label{proof_cv}

Here we will prove the abstract convergence result called theorem \ref{abstract_cv}. 
Let us begin by stating a corollary of the lemma \ref{bihari_lemma} for power functions, useful when dealing with the Lojasiewicz inequality.

\begin{lemma}[Gronwall inequality for powers]
    Let $(u_n){_n\in\mathbb{N}}, (v_n){_n\in\mathbb{N}}$ be positive sequences, $0 \leq \alpha$ such that for all $n\geq 0$:
    \begin{equation*}
        u_{n+1}-u_n \leq -v_n u_n^{\alpha}.
    \end{equation*}
    Then:
    \begin{itemize}
        \item If $\alpha>1$:
        \begin{equation*}
            u_n \leq \dfrac{1}{\left[ u_0^{1-\alpha} +(\alpha-1) \displaystyle{\sum_{k=0}^{n-1}} v_k \right]^{\frac{1}{\alpha-1}}}.
        \end{equation*}
        \item If $\alpha=1$:
        \begin{equation*}
        u_n \leq u_0 \exp{\left( -\displaystyle{\sum_{k=0}^{n-1}} v_k \right)}.
        \end{equation*}
        \item Finally, if $\alpha<1$ and the sum of $(v_k)_{k\in \mathbb{N}}$ diverges, $u_n=0$ for $n$ satisfying:
        \begin{equation*}
            \displaystyle{\sum_{k=0}^{n-1}} v_k \geq \frac{u_0^{1-\alpha}}{1-\alpha}.
        \end{equation*}
    \end{itemize}
    \label{gronwall_lemma}
\end{lemma}

The proof of the convergence Theorem \ref{abstract_cv} will follow the same steps than the GD particular case (Theorem \ref{estimate_theta}) but the existence of an accumulation point comes from Theorem \ref{LCR_accumulation} and \ref{LCM_accumulation} rather than proposition \ref{LCEGD_accumulation}.

\begin{proof}[Convergence Theorem]
    The sequence $(y_n)_{n\in \mathbb{N}}$ is bounded and admits an accumulation point $y^* \in \mathcal{Z}$ according to theorems \ref{LCR_accumulation} and \ref{LCM_accumulation}. The goal is to prove that $y_n \to y^*$.\\
    The sequence $(V(y_n))_{n\in\mathbb{N}}$ is decreasing and bounded by below by 0. So it converges to some real $l$. Without loss of generality assume that $l=V(y^*)=0$. If the sequence $(y_n)_{n\in\mathbb{N}}$ is eventually constant then the result is straightforward. Otherwise remove all the indices such that $y_{n+1}=y_n$. \\
    Now note that:
    \begin{equation*}
        V(y_{n+1})=V(y_n) \Rightarrow y_{n+1}=y_n.
    \end{equation*}
    Indeed, we have $V(y_{n+1})-V(y_n)\leq \lambda \eta_n \dot{V}(y_n) \leq 0$. As $V(y_{n+1})=V(y_n)$ it follows that $\eta_n\dot{V}(y_n)=0$. As $\eta_n$ is strictly positive it comes: $\dot{V}(y_n)=0$.  Using condition i) we can deduce that either $\nabla V(y_n)=0$ or $F(y_n)=0$. But $\nabla V=0 \Rightarrow F=0$, then $y_{n+1}=y_n$.
    As a result $V(y_n)$ is strictly decreasing and $V(y_n)>0$.\\
    
    Provided $y_n \in \mathcal{U}$ assumptions i) and ii) together with the dissipation condition gives (as $\gamma\geq 0$, $x\mapsto x^{\gamma}$ is increasing):
    \begin{equation*}
        V(y_n)-V(y_{n+1}) \geq \lambda c \|\nabla V(y_n)\| \|y_{n+1}-y_n\| \geq \lambda c c_1^{\gamma}\|y_{n+1}-y_n\| V(y_n)^{\gamma(1-\alpha_1)}.
    \end{equation*}
    Given that $V(y_n)>0$:
    \begin{equation*}
        \|y_{n+1}-y_n\| \leq \dfrac{V(y_n)-V(y_{n+1})}{\lambda c c_1^{\gamma}V(y_n)^{\gamma(1-\alpha_1)}}.
    \end{equation*}
    Since $\forall x \in \left[ V(y_{n+1}),V(y_n) \right]$, $V(y_n)^{-\gamma (1-\alpha)} \leq x^{-\gamma (1-\alpha_1)} \leq V(y_{n+1})^{-\gamma (1-\alpha_1)}$ (because $\alpha_1<1$ and $\gamma \geq 0$), we deduce:
    \begin{multline*}
    \dfrac{V(y_n)-V(y_{n+1})}{V(y_n)^{\gamma (1-\alpha_1)}} = 
    \displaystyle{\int_{V(y_{n+1})}^{V(y_n)}} \dfrac{dx}{V(y_n)^{\gamma (1-\alpha_1)}} \leq\displaystyle{\int_{V(y_{n+1})}^{V(y_n)}} \dfrac{dx}{x^{\gamma (1-\alpha_1)}} \\
    = \dfrac{1}{1-\gamma(1-\alpha_1)} \left[V(y_n)^{1-\gamma (1-\alpha_1)}- V(y_{n+1})^{1-\gamma (1-\alpha_1)}\right].
    \end{multline*}
    Since $\gamma(1-\alpha_1)\neq 1$. Provided $y_n \in \mathcal{U}$, we have:
    \begin{equation*}
        \|y_{n+1}-y_n\| \leq C \left[V(y_n)^{1-\gamma (1-\alpha_1)}- V(y_{n+1})^{1-\gamma (1-\alpha_1)}\right].
    \end{equation*}
    where $C$ is the constant defined in the theorem.
Given that $p>q$ such that $y_p, \cdots, y_{q-1} \in \mathcal{U}$, it follows:
\begin{equation*}
    \displaystyle{\sum_{n=p}^{q-1}}\|y_{n+1}-y_n\| \leq C \left[V(y_p)^{1-\gamma (1-\alpha_1)}- V(y_q)^{1-\gamma (1-\alpha_1)}\right].
\end{equation*}
Let $r>0$ be such that $B_r(y^*) \subset \mathcal{U}$. Given that $y^*$ is a accumulation point of $y_n$ and $V(y_n)^{1-\gamma (1-\alpha_1)}$ converges to $0$ since $\gamma(1-\alpha_1)<1$, it exists $n_0$ such that:
\begin{equation*}
    \begin{array}{lll}
    \|y_{n_0}-y^*\|<\frac{r}{2}, \\\\
    \forall q \geq n_0: C \left[V(y_{n_0})^{1-\gamma (1-\alpha_1)}- V(y_q)^{1-\gamma (1-\alpha_1)}\right] < \frac{r}{2}.
    \end{array}
\end{equation*}

Let us show that $y_n \in B_r(y^*)$ for all $n>n_0$. By contradiction assume that it is not the case: $\exists n>n_0, y_n \notin B_r(y^*)$. The set $\left\{ n>n_0, y_n \notin B_r(y^*)\right\}$ is a non empty bounded by below part of $\mathbb{N}$. So we can consider the minimum of this set that we denote by $p$. As a result, $\forall n_0 \leq n < p$, $y_n \in \mathcal{U}$ so:
\begin{equation*}
    \displaystyle{\sum_{n=n_0}^{p-1}}\|y_{n+1}-y_n\| \leq C \left[V(y_{n_0})^{1-\gamma (1-\alpha_1)}- V(y_q)^{1-\gamma (1-\alpha_1)}\right] < \frac{r}{2}. 
\end{equation*}
This implies:
\begin{equation*}
    \|y_p-y^*\| \leq \displaystyle{\sum_{n=n_0}^{p-1}}\|y_{n+1}-y_n\| + \|y_{n_0}-y^*\| < r.
    \label{decomposition_y}
\end{equation*}
This is a contradiction because $\|y_p-y^*\| \geq r$. As $r$ is arbitrary small this shows the convergence.\\

Now, for $n\geq n_0$:
\begin{multline}
    \label{V_Gronwall}
    \|y_n-y^*\| = \|\displaystyle{\sum_{k=n}^{+\infty}} (y_k-y_{k+1})\| \leq \displaystyle{\sum_{k=n}^{+\infty}} \|y_{k+1}-y_k\| \\ 
    \leq C \displaystyle{\lim_{q \to +\infty}} \left[V(y_{n})^{1-\gamma (1-\alpha_1)}- V(y_q)^{1-\gamma (1-\alpha_1)}\right]
    = C V(y_n)^{1-\gamma (1-\alpha_1)}.
\end{multline}
Combining the inequalities (i) and \eqref{condition_vitesse} with the dissipation inequality \eqref{dissipation_inequality} leads, for $n\geq n_0$, to:
    \begin{multline*}
      V(y_{n+1}) - V(y_n) \leq \lambda \eta_n \dot{V}(y_n) \leq -\lambda c c_1^{\gamma} \eta_n \|F(y_n)\| V(y_n)^{\gamma(1-\alpha_1)} \\ 
      \leq -\lambda c c_1^{\gamma}c_2 \eta_n V(y_n)^{\gamma(1-\alpha_1)+1-\alpha_2}.
    \end{multline*}
    Applying the lemma \ref{gronwall_lemma} to the last inequality gives an estimation of $V(y_n)$. Combining this estimation with the inequality \eqref{V_Gronwall} ($x\mapsto x^{1-\gamma(1-\alpha_1)}$ is increasing) gives the expected convergence rates on $\|\theta_n-\theta^*\|$.
\end{proof}

\end{appendices}

\clearpage


\bibliography{sn-bibliography}

\end{document}